\documentclass{article}

\usepackage{amsmath}
\usepackage{graphicx}
\usepackage{float}
\usepackage{enumerate}
\usepackage{amssymb}
\usepackage{amsfonts}
\usepackage{mathtools}
\usepackage{xcolor}
\usepackage{amsmath}
\usepackage{algorithm,algpseudocode}
\usepackage{subcaption}
\usepackage{caption}
\usepackage{verbatim}
\usepackage{bm}
\usepackage{amsthm} 
\usepackage[toc,page]{appendix}
\usepackage{hyperref} 
\usepackage{cleveref}
\usepackage[export]{adjustbox}
\usepackage{rotating}
\usepackage{multirow}
\usepackage{diagbox}
\usepackage{stmaryrd}
\usepackage[utf8]{inputenc}
\usepackage{xstring}
\usepackage{ifthen}
\usepackage{comment}
\usepackage[normalem]{ulem}

\usepackage{cancel}

\usepackage{pgf,tikz,pgfplots}
\usetikzlibrary{decorations.markings}
\pgfplotsset{compat=1.15}
\usepackage{mathrsfs}
\usetikzlibrary{arrows}

\newboolean{hidebool}
\setboolean{hidebool}{false}

\ifthenelse{\boolean{hidebool}}
  {
   \newcommand{\hide}[1]{}%
  }
  {
   \newcommand{\hide}[1]{#1}%
  }

\DeclareMathOperator{\tr}{tr}

\DeclareMathOperator{\rank}{rank}
\DeclareMathOperator{\diag}{diag}
\DeclareMathOperator{\range}{range}

\renewcommand*{\eqref}[1]{(\ref{#1})}

\newtheorem{theorem}{Theorem}[section]
\newtheorem{corollary}[theorem]{Corollary}
\newtheorem{lemma}[theorem]{Lemma}

\newtheorem{setting}{Setting}[section]

\theoremstyle{definition}

\title{Randomized low-rank approximation \\ of monotone matrix functions\footnote{This work has been supported by the SNSF research project \textit{Fast algorithms from low-rank updates}, grant number: 200020\_178806. Institute of Mathematics, EPF Lausanne, 1015 Lausanne, Switzerland. E-mails: \href{mailto:david.persson@epfl.ch}{david.persson@epfl.ch}, \href{mailto:daniel.kressner@epfl.ch}{daniel.kressner@epfl.ch}}}

\author{David Persson\footnotemark[1] \and Daniel Kressner\footnotemark[1]}

\begin{document}

\maketitle

\begin{abstract}
This work is concerned with computing low-rank approximations of a matrix function $f(\bm{A})$ for a large symmetric positive semi-definite matrix $\bm{A}$, a task that arises in, e.g., statistical learning and inverse problems. The application of popular
randomized methods, such as the randomized singular value decomposition or the Nyström approximation, to $f(\bm{A})$ 
requires multiplying $f(\bm{A})$ with a few random vectors. A significant disadvantage of such an approach, matrix-vector products with $f(\bm{A})$ are considerably more expensive than matrix-vector products with $\bm{A}$, even when carried out only approximately via, e.g., the Lanczos method.
In this work, we present and analyze funNystr\"om, a simple and inexpensive method that constructs a low-rank approximation of $f(\bm{A})$ directly from a Nystr\"om approximation of $\bm{A}$, completely bypassing the need for matrix-vector products with $f(\bm{A})$. It is sensible to use funNystr\"om whenever $f$ is monotone and satisfies $f(0) = 0$.
Under the stronger assumption that $f$ is operator monotone, which includes the matrix square root $\bm{A}^{1/2}$ and the matrix logarithm $\log(\bm{I}+\bm{A})$,
we derive probabilistic bounds for the error in the Frobenius, nuclear, and operator norms. These bounds confirm the numerical observation that funNystr\"om tends to return an approximation that compares well with  the best low-rank approximation  of 
$f(\bm{A})$. Furthermore, compared to existing methods, funNyström requires significantly fewer matrix-vector products with $\bm{A}$ to obtain a low-rank approximation of $f(\bm{A})$, without sacrificing accuracy or reliability.
 Our method is also of interest when estimating quantities associated with $f(\bm{A})$, such as the trace or the diagonal entries of $f(\bm{A})$. In particular, we propose and analyze funNystr\"om++, a combination of funNystr\"om with the recently developed Hutch++ method for trace estimation.
\end{abstract}



    

\section{Introduction}

Matrix functions appear in many areas of applied mathematics, such as differential equations~\cite{havel1994matrix,highamscaleandsquare,sidje1999numerical}, network analysis~\cite{estrada2000characterization,estradahigham}, statistical learning~\cite{alexanderian2014optimal,wenger2021reducing}, nuclear norm estimation~\cite{ubarusaad,ubaru2017applications}, and matrix equations~\cite{roberts1980linear}. Given a symmetric matrix $\bm{A} \in \mathbb{R}^{n \times n}$ with spectral decomposition $\bm{A} = \bm{U} \bm{\Lambda} \bm{U}^T$, and a function $f$ defined on the eigenvalues of $\bm{A}$, the matrix function $f(\bm{A})$ is defined as
\begin{equation*}
    f(\bm{A}) = \bm{U} f(\bm{\Lambda}) \bm{U}^T, \quad f(\bm{\Lambda}) = \begin{bmatrix} f(\lambda_1) & & \\ 
    & \ddots & \\
    & & f(\lambda_n) \end{bmatrix}.
\end{equation*}
Explicitly computing $f(\bm{A})$ via this definition requires $O(n^3)$ operations using standard algorithms~\cite{matrixcomputations}, which becomes prohibitively expensive for larger $n$.

Quite a few applications of matrix functions only require quantities associated with $f(\bm{A})$ instead of the full matrix function. Notable examples include the trace $\tr(f(\bm{A}))$ and the diagonal elements of $f(\bm{A})$, which can be estimated with Monte Carlo methods~\cite{baston2022stochastic,diagest,cortinoviskressner,hallman2022monte,ubarusaad,ubaru2017applications}. In recent years, there has been increased attention to the use of randomized low-rank approximation techniques in this context, for estimating these quantities~\cite{lizhu,saibaba} or as a variance reduction technique for Monte Carlo methods~\cite{chen,jiang2021optimal,hutchpp,AHutchpp}. These techniques facilitate matrix-vector products (mvps) with random vectors to construct cheap, yet accurate, approximations of matrices with small numerical rank, see~\cite{gittensmahoney,rsvd,nakatsukasa2020fast,tropp2017fixed} for examples. Applied to a matrix function, randomized low-rank approximation requires to perform mvps with $f(\bm{A})$, a nontrival task for large $n$. Usually only $\bm{A}$ itself can be accessed directly via mvps and one needs to resort to approximate methods, such as Lanczos~\cite[Chapter 13]{highamfunctions} and block Lanczos methods~\cite{block_lanczos}. As the error of Lanczos is linked to  polynomial approximations of $f$~\cite[Proposition 6.3]{saaditerative}, one may observe slow convergence for ``difficult'' functions $f$, for example when $f$ has a singularity close to the eigenvalues of $\bm{A}$. To avoid this, one could resort to rational approximation such as rational Krylov subspace methods \cite{guttel}, but these methods require the solution of a (shifted) linear system with $\bm{A}$ in every iteration, which comes with challenges on its own.
Even when Lanczos converges quickly and only requires a few iterations to reach the desired accuracy, the corresponding cost for approximating $f(\bm{A})$ can still be significantly more expensive than obtaining a randomized low-rank approximation of $\bm{A}$ itself. %

For certain functions $f$ it is possible to obtain a low-rank approximation of $f(\bm{A})$ directly from a low-rank approximation of $\bm{A}$, a key observation that potentially allows us to completely bypass the need for performing mvps with $f(\bm{A})$. The following basic lemma provides a first result in this direction, for the special case of \emph{best} low-rank approximations (with respect to a unitarily invariant norm).  
\begin{lemma}\label{lemma:rankpreserve}
Given a symmetric positive semi-definite (SPSD) matrix $\bm{A} \in \mathbb{R}^{n \times n}$ consider the best rank-$k$ approximation $\bm{U}_1 \bm{\Lambda}_1 \bm{U}_1^T$, where 
 $\bm{\Lambda}_1\in \mathbb{R}^{k \times k}$ is a diagonal matrix containing the $k$ largest eigenvalues of $\bm{A}$ on the diagonal and
$\bm{U}_1 \in \mathbb{R}^{n \times k}$ contains an orthonormal basis of the corresponding eigenvectors. 
Then, for monotonically increasing $f: [0,\infty) \mapsto [0,\infty)$ it holds that 
$\bm{U}_1 f(\bm{\Lambda}_1) \bm{U}_1^T$ is a best rank-$k$ approximation of $f(\bm{A})$. If, in addition, $f(0) = 0$ then $f(\bm{U}_1 \bm{\Lambda}_1 \bm{U}_1^{T}) = \bm{U}_1f(\bm{\Lambda}_1)\bm{U}_1^T$.
 \end{lemma}
 \begin{proof}
 By the spectral decomposition of $\bm A$, we can write $\bm A = \bm{U}_1 \bm{\Lambda}_1 \bm{U}_1^T + \bm{U}_2 \bm{\Lambda}_2 \bm{U}_2^T$, with diagonal $\bm \Lambda_2$ and orthonormal $\bm U_2$. Because the first term is a best low-rank approximation, none of the eigenvalues of $\bm \Lambda_2$ is larger than any of the eigenvalues of $\bm \Lambda_1$. Because of monotonicity, the same statement holds for the relation between the eigenvalues of $f(\bm \Lambda_2)$ and $f(\bm \Lambda_1)$. Using the spectral decomposition $f(\bm{A}) = \bm{U}_1 f(\bm{\Lambda}_1) \bm{U}_1^T + \bm{U}_2 f(\bm{\Lambda}_2) \bm{U}_2^T$, this implies the first claim of the lemma. The final claim is a consequence of $f(0) = 0$, since $f(\bm{U}_1 \bm{\Lambda}_1 \bm{U}_1^T) = \bm{U}_1 f(\bm{\Lambda}_1) \bm{U}_1^T + f(0)\bm{U}_2\bm{U}_2^T = \bm{U}_1 f(\bm{\Lambda}_1) \bm{U}_1^T$.
 \end{proof}
%

The result of Lemma~\ref{lemma:rankpreserve} is constrained to \emph{best} rank-$k$ approximations and does not extend to the quasi-optimal rank-$k$ approximations $\widehat{\bm{A}}$ of $\bm{A}$ usually returned by (randomized) numerical algorithms. One still hopes that 
\begin{equation}\label{eq:quasioptimal}
    \|f(\bm{A}) - f(\widehat{\bm{A}})\|
\end{equation}
is small when $\|\bm{A} - \widehat{\bm{A}}\|$ is small for a unitarily invariant norm $\|\cdot\|$. A similar idea was used in \cite{saibaba}, which analyzes approximations of the form 
$
    \tr(\log(\bm{I} + \widehat{\bm{A}})) \approx \tr(\log(\bm{I} + \bm{A}));
$ see also \cite{lizhu}. 

In this work, we will derive bounds for the approximation~\eqref{eq:quasioptimal} when $f$ is an \emph{operator monotone} function satisfying $f(0) = 0$. A function $f$ is called operator monotone if $\bm{B} \succeq \bm{C}$ for symmetric $\bm{B}, \bm{C} \in \mathbb{R}^{n \times n}$ implies $f(\bm{B}) \succeq f(\bm{C})$ \cite[p.112]{bhatia}, where $\bm{B} \succeq \bm{C}$ means that $\bm{B}-\bm{C}$ is SPSD; see \cite[Definition 7.7.1]{matrixanalysis}. Trivially, any operator monotone function  is monotonically increasing, but the converse is not true. For example, the functions $\exp(x)-1$ and $x^2$ are monotonically increasing on $[0,\infty)$ but not operator monotone. Examples of operator monotone functions include $\sqrt{x}, \log(1+x)$ and $\frac{x}{x + \lambda}$ for $\lambda > 0$; see~\cite[Sec. V.1]{bhatia}. In the following, we briefly highlight a few of the numerous applications for these functions. 

\subsubsection*{Matrix square root}
Matrix square roots of SPSD matrices play an important role in sampling from Gaussian distributions. If $\bm{\omega} \sim N(\bm{0},\bm{I}_n)$ and $\bm{A} = \bm{G}\bm{G}^T$ then $\bm{\psi} = \bm{\mu} + \bm{G}\bm{\omega} \sim N(\bm{\mu},\bm{A})$, and  a possible choice is $\bm{G} = \bm{A}^{1/2}$ \cite{pleiss2020fast}. Having a low-rank approximation $\widehat{\bm{A}}^{1/2}$ of $\bm{A}^{1/2}$ at hand allows us to cheaply sample from $\widehat{\bm{\psi}} = \bm{\mu} + \widehat{\bm{A}}^{1/2} \bm{\omega} \sim N(\bm{\mu},\widehat{\bm{A}})$. Moreover, the mean-squared error remains small for an accurate approximation $\widehat{\bm{A}}^{1/2}$: 
\begin{equation*}
    \mathbb{E}\|\bm{\psi}-\widehat{\bm{\psi}}\|_2^2 = \|\bm{A}^{1/2}-\widehat{\bm{A}}^{1/2}\|_F^2.
\end{equation*}
The same technique can be used to sample from a general elliptical distribution~\cite{quantrisk}.\footnote{For a general elliptical distribution the mean-squared error equals $c\|\bm{A}^{1/2}-\widehat{\bm{A}}^{1/2}\|_F^2$ where $c$ is a constant depending on the elliptical distribution.} 

The matrix square root also appears when estimating the nuclear norm $\|\bm{X}\|_*$ of a matrix~\cite{ubarusaad,ubaru2017applications}. Because of the relation
\begin{equation*}
    \|\bm{X}\|_* = \tr(\bm{A}^{1/2}), \quad \text{where } \bm{A} = \bm{X}^T \bm{X},
\end{equation*}
nuclear norm estimation is equivalent to performing trace estimation on $\bm{A}^{1/2}$, a task that benefits from low-rank approximation. 

\subsubsection*{Matrix logarithm}
The matrix function $\log(\bm{I}+\bm{A})$ frequently appears in statistical learning~\cite{gardner2018gpytorch}. In these applications one typically aims at estimating $\log\det(\bm{I} + \bm{A}) = \tr\left(\log(\bm{I} + \bm{A})\right)$, a task that greatly benefits from low-rank approximation of $\log(\bm{I} + \bm{A})$.

\subsubsection*{Effective dimension}
The effective dimension $d_{\text{eff}}(\mu)$, also called statistical dimension, is defined as
\begin{equation*}
    d_{\text{eff}}(\mu) = \tr(f_{\mu}(\bm{A})), \quad f_{\mu}(x) = \frac{x}{x+\mu},\quad \mu > 0.
\end{equation*}
This quantity appears in kernel learning \cite{alaoui2015fast,avronfasterkernel,avronsharper} and inverse problems \cite{meier2022randomized}. The effective dimension can once again be estimated using trace estimation, and low-rank approximation is again beneficial. Another important quantity is the diagonal of $f_{\mu}(\bm{A})$; its entries are called the Ridge leverage scores.

\subsection{Contributions}

In this work we present funNystr\"om, a new and simple method to obtain a randomized low-rank approximation of $f(\bm{A})$ for an SPSD matrix $\bm A$ and an increasing function $f:[0,\infty) \mapsto [0,\infty)$ satisfying $f(0) = 0$. In a nutshell, our method returns $f(\widehat{\bm{A}})$, where $\widehat{\bm{A}}$ is a Nyström approximation~\cite{gittensmahoney} of $\bm A$.
A major advantage of funNystr\"om is that it only requires mvps with $\bm A$ and not with $f(\bm{A})$. Furthermore, our method can be made single-pass, meaning that we only need to access the entries of $\bm A$ once.
This property does not hold for the Nyström approximation applied to $f(\bm{A})$; the Lanczos method for approximating the involved mvps with $f(\bm A)$ usually requires several iterations and thus repeatedly accesses $\bm{A}$.

For operator monotone $f$ we derive (probabilistic) bounds on
\begin{equation}\label{eq:error}
    \|f(\bm{A})-f(\widehat{\bm{A}})\|
\end{equation}
for different unitarily invariant norms $\|\cdot \|$, including the Frobenius norm $\|\cdot\|_F$, the nuclear norm $\|\cdot\|_*$, and the operator norm $\|\cdot\|_2$. 
The bounds predict that funNystr\"om achieves similar accuracy compared to the (much more expensive) alternative of applying randomized low-rank approximation to $f(\bm{A})$. Let us emphasize that its inexpensiveness makes funNystr\"om also relevant when only quantities like the trace $\tr(f(\bm{A}))$ instead of the full matrix function are needed. Our bounds not only improve and generalize the error bounds in~\cite{saibaba} on trace estimation for $f(x) = \log(1+x)$ but they also show how funNystr\"om can be used to improve the recently developed Hutch++ algorithm~\cite{hutchpp} for trace estimation in the context of matrix functions.

The numerical experiments reported in this work confirm our theoretical findings; compared to existing methods funNystr\"om requires fewer mvps with $\bm{A}$ to obtain a low-rank approximation or trace estimate for $f(\bm{A})$, without sacrificing accuracy or reliability.

\

\section{funNystr\"om: Randomized low-rank approximation of matrix functions}


Let $\bm{\Omega}$ denote an $n \times k$ Gaussian random matrix, that is, its entries are i.i.d. standard normal random variables. 
Given an $n\times n$ SPSD matrix $\bm{A}$, we consider the Nyström approximation 
\begin{equation}\label{eq:nystrom}
    \widehat{\bm{A}}_{q,k} := \bm{A}^q \bm{\Omega}(\bm{\Omega}^T \bm{A}^{2q-1} \bm{\Omega})^{\dagger}(\bm{A}^q\bm{\Omega})^T,
\end{equation} 
where $(\cdot)^\dagger$ denotes the Moore-Penrose pseudoinverse of a matrix and $q$ is a small integer, say, $q = 1$ or $q = 2$.

Since $\bm{\Omega}^T \bm{A}^{2q-1} \bm{\Omega}$ is SPSD, we can compute its square root $\bm{M}$ to obtain
\[
 \widehat{\bm{A}}_{q,k} = \big( \bm{A}^q \bm{\Omega} \bm{M}^{\dagger} \big) \big( \bm{A}^q \bm{\Omega} \bm{M}^{\dagger} \big)^T.
\]
The singular value decomposition (SVD) of the factor, $\bm{A}^q \bm{\Omega} \bm{M}^{\dagger} = \widehat{\bm{U}} \bm{\Sigma} \bm{V}^T$, yields a (truncated) spectral decomposition $\widehat{\bm{A}}_{q,k} = \widehat{\bm{U}} \widehat{\bm{\Lambda}} \widehat{\bm{U}}^T$ with $\widehat{\bm{\Lambda}} = \bm{\Sigma}^2$. Using the assumption $f(0) = 0$, it follows that \[f(\widehat{\bm{A}}_{q,k}) = 
\widehat{\bm{U}} f(\widehat{\bm{\Lambda}}) \widehat{\bm{U}}^T.\]
Algorithm~\ref{alg:funnystrom} implements this idea. Some remarks on our implementation of the algorithm:
\begin{itemize}
    \item To improve numerical stability, the subspace iteration in line~\ref{line:subspace iteration} is combined with a QR decomposition at each iteration; see \cite[Section 5.1]{saadnumerical}. 
    This replaces $\bm{A}^{q-1} \bm{\Omega}$ by an orthonormal basis for its range, which does not alter the Nyström approximation~\cite[Proposition 2.7]{tropp2017fixed}.
    \item To mitigate numerical issues when $\bm{Q}^T \bm{Y}$ in line \ref{line:cholesky} is highly ill-conditioned, in our implementation we truncate eigenvalues in $\bm{D}$ smaller than $5\cdot 10^{-16}\cdot\|\bm{D}\|_2$ to $0$.
\end{itemize}

\begin{algorithm}
\caption{funNyström}
\label{alg:funnystrom}
\textbf{input:} SPSD $\bm{A} \in \mathbb{R}^{n \times n}$. Rank $k$. Number of subspace iterations $q$. Increasing function $f:[0,\infty) \mapsto [0,\infty)$ satisfying $f(0)=0$.\\
\textbf{output:} Spectral decomposition of $f(\widehat{\bm{A}}_{q,k}) = \widehat{\bm{U}} f(\widehat{\bm{\Lambda}})\widehat{\bm{U}}^T$
\begin{algorithmic}[1]
    \State Sample $\bm{\Omega} \in \mathbb{R}^{n \times k}$ with i.i.d. $N(0,1)$ entries.
    \State Compute a thin QR decomposition of $\bm{\Omega} = \bm{Q}\bm{R}$.
    \For{$q_{\text{count}} = 1,\ldots,q-1$}\label{line:forloop}
    \State $\bm{X} = \bm{A} \bm{Q}$\label{line:subspace iteration}
    \State Compute a thin QR decomposition of $\bm{X} = \bm{Q}\bm{R}$.
    \EndFor
    \State $\bm{Y} = \bm{A}\bm{Q}$ \label{line:Y}
    \State Compute the eigenvalue decomposition of $\bm{Q}^T \bm{Y} = \bm{V} \bm{D} \bm{V}^T$. \label{line:cholesky}
    \State $\bm{B} = \bm{Y}\bm{V} (\bm{D}^{1/2})^{\dagger} \bm{V}^T$
    \State Compute the SVD of $\bm{B} = \widehat{\bm{U}} \bm{\Sigma} \bm{V}^T$.
    \State $\widehat{\bm{\Lambda}} = \bm{\Sigma}^2$ \label{line:lambda}
    \State \textbf{return} $\widehat{\bm{U}}, f(\widehat{\bm{\Lambda}})$.
\end{algorithmic}
\end{algorithm}
For $q = 1$, the loop in line~\ref{line:forloop} of Algorithm~\ref{alg:funnystrom} becomes empty and Algorithm~\ref{alg:funnystrom} requires a single pass over $\bm{A}$. Moreover, the mvps with $\bm{A}$, which usually constitute the dominant cost of the method, can be carried out entirely in parallel. 

There is a variant of Algorithm~\ref{alg:funnystrom} that uses the randomized SVD of $\bm{A}$ instead of the Nyström approximation. This variant first computes an orthonormal basis $\bm{Q} \in \mathbb{R}^{n \times k}$ of $\range(\bm{A}^q \bm{\Omega})$ and then constructs a (truncated) spectral decomposition of $\bm{Q}(\bm{Q}^T \bm{A} \bm{Q}) \bm{Q}^T$. This variant requires $(q+1)k$ mvps with $\bm{A}$. For the same number of mvps, Algorithm~\ref{alg:funnystrom} can carry out an additional subspace iteration. The corresponding Nystr\"om approximation $\widehat{\bm{A}}_{q+1,k}$ is often significantly better than the one obtained from the randomized SVD with $q$ subspace iterations~\cite{gittensmahoney,rsvd,nakatsukasa2020fast}. Therefore, for the case of an SPSD matrix $\bm{A}$ using the Nyström approximation is often preferred.

\section{Error bounds for operator monotone functions}

In this section, we will derive bounds for the approximation error \eqref{eq:error} in the Frobenius norm, nuclear norm, and the operator norm. Each of these three norm settings is of importance on its own in applications. For example, the Frobenius norm error dictates the mean square error when approximating samples from elliptical distributions and the variance of the stochastic trace estimator after using a low-rank approximation to reduce the variance of the estimator~\cite{hutchpp}. The nuclear norm is important when approximating $\tr(f(\bm{A}))$ with $\tr(f(\widehat{\bm{A}}_{q,k}))$, since by the operator monotonicity of $f$ the nuclear norm error simplifies to the trace error; see Section~\ref{section:generalresults} and Section~\ref{section:nuclearnormerrorbounds} for details and Section~\ref{section:trest} for applications of our bounds in trace estimation. Finally, the operator norm is often the most natural norm to consider when $\bm{A}$ represents a (discretized) linear operator. 

The following setting will be assumed for the main results presented in this section.

\begin{setting}\label{setting}
For SPSD $\bm{A} \in \mathbb{R}^{n \times n}$ with eigenvalues $\lambda_1 \geq \lambda_2 \geq \ldots \geq \lambda_n \geq 0$, consider the spectral decomposition
\begin{equation} \label{eq:spectraldecom}
    \bm{A} = \bm{U} \bm{\Lambda} \bm{U}^T = \begin{bmatrix} \bm{U}_1 & \bm{U}_2 \end{bmatrix} \begin{bmatrix} \bm{\Lambda}_1 & \\ & \bm{\Lambda}_2 \end{bmatrix} \begin{bmatrix} \bm{U}_1^T \\ \bm{U}_2^T \end{bmatrix},
\end{equation}
where $\bm{\Lambda}_1 = \diag(\lambda_1,\ldots,\lambda_k)$ and 
$\bm{\Lambda}_2 = \diag(\lambda_{k+1},\ldots,\lambda_n)$.
We assume $\lambda_k>0$ and let
\begin{itemize}
    \item $\gamma = \frac{\lambda_{k+1}}{\lambda_k}$ denote the $k^{\text{th}}$ spectral gap;
    \item $p \geq 0$ be an oversampling parameter;
    \item $\bm{\Omega}$ be an $n \times (k+p)$ Gaussian random matrix;
    \item $\bm{\Omega}_i = \bm{U}_i^T \bm{\Omega}$ for $i = 1,2$;
    \item $\widehat{\bm{A}}_{q,k+p}$ be the rank-$(k+p)$ Nyström approximation defined in \eqref{eq:nystrom} for $q\ge 1$;
    \item $f:[0,\infty) \mapsto [0,\infty)$ be 
    operator monotone and $f(0) = 0$ .
\end{itemize}
\end{setting}
\noindent Some remarks on Setting~\ref{setting}:
\begin{itemize}
    \item By the unitary invariance of Gaussian random vectors, $\bm{\Omega}_1$ and $\bm{\Omega}_2$ are independent Gaussian random matrices. This implies that $\rank(\bm{\Omega}_1) = k$ holds almost surely, in which case $\bm{\Omega}_1^{\dagger} = \bm{\Omega}_1^T(\bm{\Omega}_1\bm{\Omega}_1^T)^{-1}$.
    \item The assumption $\lambda_k > 0$ has been made for convenience. When $\lambda_k = 0$ then the approximation results derived in this work become trivial because, with probability one, $\widehat{\bm{A}}_{q,k+p} = \bm{A}$ and the approximation error \eqref{eq:error} is 0 in this case.\footnote{To see this, recall $\widehat{\bm{A}}_{q,k+p} = \bm{A}^{1/2} \bm{P}_{\bm{A}^{q-1/2}\bm{\Omega}} \bm{A}^{1/2}$, where $\bm{P}_{\bm{A}^{q-1/2}\bm{\Omega}}$ is the orthogonal projector onto $\range(\bm{A}^{q-1/2}\bm{\Omega})$. However, as argued in the proof of \cite[Theorem 9.1]{rsvd}, if $\lambda_k = 0$ then $\range(\bm{A}^{q-1/2}\bm{\Omega}) = \range(\bm{A}^{q-1/2}) = \range(\bm{A})$ almost surely. Thus, $\bm{P}_{\bm{A}^{q-1/2}\bm{\Omega}} = \bm{P}_{\bm{A}}$ almost surely. Consequently, $\widehat{\bm{A}}_{q,k+p} = \bm{A}^{1/2} \bm{P}_{\bm{A}} \bm{A}^{1/2} = \bm{A}$ almost surely.} 
    \item 
    The operator monotonicity of $f$ implies that $f \in C^{\infty}(0,\infty)$ and that $f$ is (right) continuous at $0$; see \cite[p.134-135]{bhatia}. The operator monotonicity of $f$ is equivalent to operator concavity of $f$ \cite[Theorem V.2.5]{bhatia}.
\end{itemize}

\subsection{General results}\label{section:generalresults}

We start our analysis by collecting preliminary results that do not depend on the choice of norm. For this purpose, 
let us recall from~\cite[Lemma 1]{gittensmahoney} that we can write
\begin{equation} \label{eq:nystromproj}
  \widehat{\bm{A}}_{q,k+p} = \bm A^{1/2} \bm{P}_{\bm{A}^{q-1/2}\bm{\Omega}} \bm A^{1/2},
\end{equation}
where we let $\bm{P}_{\bm{M}}$ denote the orthogonal projection onto the range of a matrix $\bm M$.
This implies
\[
 \bm A - \widehat{\bm{A}}_{q,k+p} = \bm A^{1/2} ( \bm I - \bm{P}_{\bm{A}^{q-1/2}\bm{\Omega}} ) \bm A^{1/2} \succeq 0.
\]
Combined with operator monotonicity, it follows that $f(\bm A)$ is approximated from below:
\begin{equation} \label{eq:approxfrombelow}
 f(\widehat{\bm{A}}_{q,k+p}) \preceq f(\bm A).
\end{equation}

The following results hold for any unitarily invariant norm $\|\cdot\|$.
\begin{lemma}\label{lemma:loewner_trace}
Consider $n\times n$ SPSD matrices $\bm B,\bm C$ satisfying $\bm{B} \succeq \bm{C}$. Then
\begin{enumerate}
 \item[(i)] $\|\bm B\| \ge \|\bm C\|$;
 \item[(ii)] $\|f(\bm B)\| \ge \|f(\bm C)\|$ for any increasing function $f:[0,\infty) \to [0,\infty)$;
 \item[(iii)]
 $
    \|f(\bm{B})-f(\bm{C})\| \leq \|f(\bm{B}-\bm{C})\|,
$
for any operator monotone function $f : [0,\infty)\to [0,\infty)$.
\end{enumerate}
\end{lemma}
\begin{proof}
(i)
Let $\lambda_i(\bm{B})$ and $\lambda_i(\bm{C})$ denote the $i$th largest eigenvalues of $\bm B$ and $\bm C$, respectively.
By~\cite[Corollary 7.7.4 (c)]{matrixanalysis}, $\lambda_i(\bm{B}) \geq \lambda_i(\bm{C}) \ge 0$ for $i = 1,\ldots, n$. By 
Fan's dominance theorem~\cite[Theorem IV.2.2]{bhatia}, this implies $\|\bm B\| \ge \|\bm C\|$.

(ii) Because of the monotonicity and non-negativity of $f$, $\lambda_i(f(\bm{B})) \geq \lambda_i(f(\bm{C})) \ge 0$ and hence the arguments from (i) apply.

(iii) This is a consequence of a result by Ando~\cite[Theorem 1]{ando}. \end{proof}
Lemma~\ref{lemma:loewner_trace} (ii) combined with~\eqref{eq:approxfrombelow} yields $\|f(\widehat{\bm{A}}_{q,k+p})\| \le \|f(\bm A)\|$. The following result establishes a lower bound on $\|f(\widehat{\bm{A}}_{q,k+p})\|$ in terms of the projected matrix function.
\begin{lemma}\label{lemma:move_out_projection}
Under Setting~\ref{setting} we have
\begin{equation} \label{eq:move_out_projection}
    \|f(\widehat{\bm{A}}_{q,k+p})\| = \|f(\bm{P}_{\bm{A}^{q-1/2}\bm{\Omega}}  \bm{A}\bm{P}_{\bm{A}^{q-1/2}\bm{\Omega}})\| \geq \|\bm{P}_{\bm{A}^{q-1/2}\bm{\Omega}}  f(\bm{A})\bm{P}_{\bm{A}^{q-1/2}\bm{\Omega}}\|.
\end{equation}
\end{lemma}
\begin{proof}
By~\eqref{eq:nystromproj}, $\widehat{\bm{A}}_{q,k+p} = \bm{A}^{1/2}\bm{P}_{\bm{A}^{q-1/2}\bm{\Omega}}\bm{A}^{1/2} = \bm B \bm B^T$ for $\bm B = \bm{A}^{1/2}\bm{P}_{\bm{A}^{q-1/2}\bm{\Omega}}$. The equality in~\eqref{eq:move_out_projection} follows from the fact 
that the eigenvalues of $\bm B \bm B^T$ and $\bm B^T \bm B = \bm{P}_{\bm{A}^{q-1/2}\bm{\Omega}}  \bm{A}\bm{P}_{\bm{A}^{q-1/2}\bm{\Omega}}$ are identical.
To establish the inequality in~\eqref{eq:move_out_projection}, we apply~\cite[Theorem V.2.3 (iv)]{bhatia} to $-f$ (which is operator convex~\cite[Theorem V.2.5]{bhatia}) and obtain
\begin{equation*}
    f(\bm{P}_{\bm{A}^{q-1/2}\bm{\Omega}}  \bm{A}\bm{P}_{\bm{A}^{q-1/2}\bm{\Omega}}) \succeq \bm{P}_{\bm{A}^{q-1/2}\bm{\Omega}}  f(\bm{A})\bm{P}_{\bm{A}^{q-1/2}\bm{\Omega}} \succeq \bm{0}.
\end{equation*} 
This completes the proof using Lemma~\ref{lemma:loewner_trace} (i).
\end{proof}

The following bound will be useful for relating the approximation errors of the matrix and the matrix function.
\begin{lemma}\label{lemma:concave}
Under Setting~\ref{setting} the function $f(x)/x$ is decreasing and we have
\begin{equation*}
    \frac{f(\lambda_k)}{\lambda_k} \|\bm{\Lambda}_2\| \leq \|f(\bm{\Lambda}_2)\|.
\end{equation*}
\end{lemma}
\begin{proof}
Because $f$ is a differentiable concave function we have $f(y) \leq f(x) + f'(x)(y-x)$. Letting $y = 0$ and using $f(0) = 0$ gives $f'(x)x - f(x) \leq 0$. Hence, $\frac{\partial}{\partial x}\left[\frac{f(x)}{x}\right] = \frac{f'(x)x-f(x)}{x^2} \leq 0$. The second part follows from noticing that
$\frac{f(\lambda_k)}{\lambda_k} \lambda_{k+i} \leq f(\lambda_{k+i})$ for $i = 1,\ldots,n-k$.
\end{proof}

\subsection{Frobenius norm error bounds}

In this section we will establish the following probabilistic bounds for the approximation error~\eqref{eq:error} in the Frobenius norm. 
\begin{theorem}\label{theorem:frobenius_opmon_expectation}
Under Setting~\ref{setting} suppose that $k, p, q \geq 2$. Then
\begin{align*}
    &\mathbb{E}\|f(\bm{A})-f(\widehat{\bm{A}}_{q,k+p})\|_F^2 \leq \left(1 + 5 \gamma^{2(q-3/2)} \frac{k}{p-1}\right) \|f(\bm{\Lambda}_2)\|_F^2.
\end{align*}
\end{theorem}
\begin{theorem}\label{theorem:frobenius_opmon_probability}
    Under Setting~\ref{setting} suppose that $k, p > 4$ and $q \geq 2$. Then for any $u,t \geq 1$ the inequality
    \begin{align*}
        &\|f(\bm{A})-f(\widehat{\bm{A}}_{q,k+p})\|_F \leq \left(1 + \gamma^{q-3/2} \sqrt{\frac{15k}{p+1}} t\right) \|f(\bm{\Lambda}_2)\|_F + \gamma^{q-3/2}\frac{e\sqrt{5(k+p)}}{p+1}ut\|f(\bm{\Lambda}_2)\|_2
    \end{align*}
    holds with probability at least $1-2t^{-p} - e^{-u^2/2}$. 
\end{theorem}
For $q \geq 2$ our error bounds for Algorithm~\ref{alg:funnystrom} are comparable to
existing bounds~\cite[Theorem 10.5 and 10.7]{rsvd}
for the direct (and much more expensive) application of the randomized SVD to $f(\bm{A})$.\footnote{For $q = 2$, by replacing $\gamma^{q-3/2} = \gamma^{1/2}$ in Theorem~\ref{theorem:frobenius_opmon_expectation} and \ref{theorem:frobenius_opmon_probability} with $\gamma^0 = 1$ we achieve the same bounds as in \cite[Theorem 10.5 and 10.7]{rsvd} up to constants.} In the context of trace estimation, Theorem~\ref{theorem:frobenius_opmon_probability} allows us to replace the randomized SVD by Algorithm~\ref{alg:funnystrom} in Hutch++~\cite{hutchpp}; see Section~\ref{section:trest} for details. Our bounds are only valid for $q \geq 2$; in Appendix~\ref{app:schatten} we present (structural) bounds valid for any $q \geq 1$ and for general Schatten norms but under a stronger assumption on $f$. 

\subsubsection{Structural bound}

Before applying probabilistic arguments, we derive a structural bound that holds for \emph{any} sketching matrix $\bm{\Omega}$ such that $\bm{\Omega}_1 = \bm{U}_1^T \bm{\Omega}$ has full rank.
We begin with a simple error bound.
\begin{lemma}\label{lemma:upper_bound}
Under Setting~\ref{setting} we have the inequality
    \begin{equation*}
        \|f(\bm{A})-f(\widehat{\bm{A}}_{q,k+p})\|_F^2 \leq \|f(\bm{A})\|_F^2 - \|f(\widehat{\bm{A}}_{q,k+p})\|_F^2.
    \end{equation*}
\end{lemma}
\begin{proof}
Using~\eqref{eq:approxfrombelow}, $\bm{B} \succeq \bm{C}$ for $\bm{B} = f(\bm{A})$ and $\bm{C} = f(\widehat{\bm{A}}_{q,k+p})$, it follows from~\cite[Theorem 7.7.2 (a)]{matrixanalysis} that
$
    \bm{C}^2 = \bm{C}^{1/2}\bm{C} \bm{C}^{1/2} \preceq \bm{C}^{1/2}\bm{B} \bm{C}^{1/2}.
$
In turn, 
\begin{equation*}
    \|\bm{C}\|_F^2 = \tr(\bm{C}^2) \leq \tr(\bm{C}^{1/2}\bm{B} \bm{C}^{1/2}) = \tr(\bm{B}\bm{C}),
\end{equation*}
using~\cite[Corollary 7.7.4 (d)]{matrixanalysis} for the inequality. Hence, 
\begin{equation*}
    \|\bm{B}-\bm{C}\|_F^2 = \|\bm{B}\|_F^2 + \|\bm{C}\|_F^2 - 2\tr(\bm{B}\bm{C}) \leq \|\bm{B}\|_F^2 - \|\bm{C}\|_F^2,
\end{equation*}
which concludes the proof.
\end{proof}

Lemma~\ref{lemma:upper_bound} combined with Lemma~\ref{lemma:move_out_projection} allow us to bound the Frobenius norm approximation error by a projection error:
\begin{align}
    \|f(\bm{A})-f(\widehat{\bm{A}}_{q,k+p})\|_F^2  & \leq \|f(\bm{A})\|_F^2 - \|f(\widehat{\bm{A}}_{q,k+p})\|_F^2 \nonumber \\
    &\leq \|f(\bm{A})\|_F^2 - \|\bm{P}_{\bm{A}^{q-1/2}\bm{\Omega}}f(\bm{A})\bm{P}_{\bm{A}^{q-1/2}\bm{\Omega}}\|_F^2 \nonumber \\
    &=  \|f(\bm{A})-\bm{P}_{\bm{A}^{q-1/2}\bm{\Omega}}f(\bm{A})\bm{P}_{\bm{A}^{q-1/2}\bm{\Omega}}\|_F^2, \label{eq:projectionerrorbound}
\end{align}
On the other hand, because $\bm{P}_{\bm{A}^{q}\bm{\Omega}}  f(\bm{A})\bm{P}_{\bm{A}^{q}\bm{\Omega}}$ is the best projection of the (co-)range of $f(\bm A)$ on the range of ${\bm{A}^{q}\bm{\Omega}}$ and 
$\range(f(\widehat{\bm{A}}_{q,k+p})) \subseteq \range(\bm{A}^{q}\bm{\Omega})$ we also have
\begin{equation}
    \|f(\bm{A})-\bm{P}_{\bm{A}^{q}\bm{\Omega}}  f(\bm{A})\bm{P}_{\bm{A}^{q}\bm{\Omega}}\|_F \leq \|f(\bm{A})-f(\widehat{\bm{A}}_{q,k+p})\|_F. \label{eq:projectionerrorbound2}
\end{equation}
Thus, the approximation error is sandwiched between the errors of orthogonal 
projections of $f(\bm A)$ onto ${\bm{A}^{q}\bm{\Omega}}$ and 
${\bm{A}^{q-1/2}\bm{\Omega}}$. Note that the definition $\widehat{\bm{A}}_{q,k+p}$ is valid for any $q \geq 1$, even rational $q$. Therefore, by \eqref{eq:projectionerrorbound} and \eqref{eq:projectionerrorbound2} and  $\range(f(\bm{P}_{\bm{A}^{q}\bm{\Omega}}\bm{A}\bm{P}_{\bm{A}^{q}\bm{\Omega}})) \subseteq \range(\bm{A}^{q}\bm{\Omega})$ (a consequence of $f(0)=0$)
\begin{align*}
    &\|f(\bm{A})-\bm{P}_{\bm{A}^{q+1}\bm{\Omega}}  f(\bm{A})\bm{P}_{\bm{A}^{q+1}\bm{\Omega}}\|_F \leq \|f(\bm{A})-\bm{P}_{\bm{A}^{q+1/2}\bm{\Omega}}  f(\bm{A})\bm{P}_{\bm{A}^{q+1/2}\bm{\Omega}}\|_F \\
    \leq & \|f(\bm{A})-\bm{P}_{\bm{A}^{q}\bm{\Omega}}  f(\bm{A})\bm{P}_{\bm{A}^{q}\bm{\Omega}}\|_F \leq \|f(\bm{A})-  f(\bm{P}_{\bm{A}^{q}\bm{\Omega}}\bm{A}\bm{P}_{\bm{A}^{q}\bm{\Omega}})\|_F,
\end{align*}
which in turn implies
\begin{align}
\begin{split}
    &\|f(\bm{A})-f(\widehat{\bm{A}}_{q+1,k+p})\|_F \leq  \|f(\bm{A})-\bm{P}_{\bm{A}^{q}\bm{\Omega}}  f(\bm{A})\bm{P}_{\bm{A}^{q}\bm{\Omega}}\|_F\\
    \leq &\|f(\bm{A})-  f(\bm{P}_{\bm{A}^{q}\bm{\Omega}}\bm{A}\bm{P}_{\bm{A}^{q}\bm{\Omega}})\|_F = \|f(\bm{A})-  f(\bm{Q}\bm{Q}^T\bm{A}\bm{Q}\bm{Q}^T)\|_F,
\end{split}\label{eq:upper_bound}
\end{align} 
where $\bm{Q}$ is an orthonormal basis of $\range(\bm{A}^q \bm{\Omega})$. This means that the error from Algorithm~\ref{alg:funnystrom} will never exceed the error produced by the randomized SVD, using the same number of mvps with $\bm{A}$. Similar comments for $f(x) = x$ can be found in \cite{gittensmahoney,rsvd,randmethodsmatrixcomp,nakatsukasa2020fast}.

We now proceed with the main structural bound.

\begin{lemma}\label{lemma:function_opmon_structural}
Under Setting~\ref{setting} and assuming $\rank(\bm{\Omega}_1) = k$ we have
\begin{equation*}
    \|f(\bm{A})-f(\widehat{\bm{A}}_{q,k+p})\|_F^2 \leq \|f(\bm{\Lambda}_2)\|_F^2 + 5 \gamma^{2(q-3/2)} \frac{f(\lambda_k)^2}{\lambda_k^2} \|\bm{\Lambda}_2\bm{\Omega}_2 \bm{\Omega}_1^{\dagger}\|_F^2.
\end{equation*}
\end{lemma}
\begin{proof}
Letting $\bm{Y} = \bm{A}^{q-1/2}\bm{\Omega}$, the inequality~\eqref{eq:projectionerrorbound} states that $\|f(\bm{A})-f(\widehat{\bm{A}}_{q,k+p})\|_F \le 
\|f(\bm{A})-\bm{P}_{\bm{Y}}f(\bm{A})\bm{P}_{\bm{Y}}\|_F$.
Setting $\widetilde{\bm{Y}}:=\bm{U}^T \bm{Y}$ with the orthogonal factor $\bm U$ from the spectral decomposition~\eqref{eq:spectraldecom}, we have
$\bm{U}^T \bm{P}_{\bm{Y}} \bm{U} = \bm{P}_{\widetilde{\bm{Y}}}$ \cite[Proposition 8.4]{rsvd} and thus
\begin{equation*}
    \|f(\bm{A})-\bm{P}_{\bm{Y}}f(\bm{A})\bm{P}_{\bm{Y}}\|_F^2 = \|f(\bm{\Lambda}) - \bm{P}_{\widetilde{\bm{Y}}}f(\bm{\Lambda})\bm{P}_{\widetilde{\bm{Y}}}\|_F^2.
\end{equation*}

Setting
\begin{equation*}
    \bm{Z} = \widetilde{\bm{Y}} \bm{\Omega}_1^{\dagger} \bm{\Lambda}_1^{-(q-1/2)} = \begin{bmatrix} \bm{I} \\ \bm{F} \end{bmatrix}, \quad \bm{F} = \bm{\Lambda}_2^{q-1/2} \bm{\Omega}_2\bm{\Omega}_1^{\dagger} \bm{\Lambda}_1^{-(q-1/2)}.
\end{equation*}
we note that $\range(\bm{Z}) \subseteq \range(\bm{\widetilde{Y}})$ and, in turn, $\bm{P}_{\bm{Z}}  = \bm{P}_{\bm{Z}} \bm{P}_{\widetilde{\bm{Y}}}$. Consequently, $\|\bm{P}_{\bm{Z}}f(\bm{\Lambda})\bm{P}_{\bm{Z}}\|_F = \|\bm{P}_{\bm{Z}} \bm{P}_{\widetilde{\bm{Y}}}f(\bm{\Lambda})\bm{P}_{\widetilde{\bm{Y}}} \bm{P}_{\bm{Z}}\|_F \le \|\bm{P}_{\widetilde{\bm{Y}}}f(\bm{\Lambda})\bm{P}_{\widetilde{\bm{Y}}}\|_F$ and therefore
\begin{align}
    \|f(\bm{A})-f(\widehat{\bm{A}}_{q,k+p})\|^2_F & \le \|f(\bm{\Lambda}) - \bm{P}_{\widetilde{\bm{Y}}}f(\bm{\Lambda})\bm{P}_{\widetilde{\bm{Y}}}\|_F^2 
     = \|f(\bm{\Lambda})\|_F^2 - \|\bm{P}_{\widetilde{\bm{Y}}}f(\bm{\Lambda})\bm{P}_{\widetilde{\bm{Y}}}\|_F^2 \nonumber \\
    & \leq  \|f(\bm{\Lambda})\|_F^2 -  \|\bm{P}_{\bm{Z}}f(\bm{\Lambda})\bm{P}_{\bm{Z}}\|_F^2 = \|f(\bm{\Lambda}) -  \bm{P}_{\bm{Z}}f(\bm{\Lambda})\bm{P}_{\bm{Z}}\|_F^2 \nonumber \\
    &= \|(\bm{I}-\bm{P}_{\bm{Z}})f(\bm{\Lambda})\|_F^2 + \|(\bm{I}-\bm{P}_{\bm{Z}})f(\bm{\Lambda})\bm{P}_{\bm{Z}}\|_F^2. \label{eq:decomperror}
\end{align}

For treating the first term in the sum~\eqref{eq:decomperror}, we recall from~\cite[Proposition 8.2]{rsvd} that 
\begin{align}
    &\bm{I}-\bm{P}_{\bm{Z}} = \begin{bmatrix} \bm{I}-(\bm{I} + \bm{F}^T \bm{F})^{-1} & - (\bm{I}+\bm{F}^T \bm{F})^{-1} \bm{F}^T \\
    -\bm{F}(\bm{I}+\bm{F}^T \bm{F})^{-1} & \bm{I} - \bm{F}(\bm{I}+\bm{F}^T \bm{F})^{-1} \bm{F}^T \end{bmatrix},\nonumber\\
    &\bm{I}-(\bm{I} + \bm{F}^T \bm{F})^{-1} \preceq \bm{F}^T \bm{F}, \label{eq:projector_ineq}\\
    &\bm{I} - \bm{F}(\bm{I}+\bm{F}^T \bm{F})^{-1} \bm{F}^T \preceq \bm{I}\nonumber.
\end{align}
Hence,
\begin{equation}\label{eq:projection}
     \|(\bm{I}-\bm{P}_{\bm{Z}})f(\bm{\Lambda})\|_F^2 = \tr(f(\bm{\Lambda})(\bm{I}-\bm{P}_{\bm{Z}})f(\bm{\Lambda})) \leq \|\bm{F}f(\bm{\Lambda}_1)\|_F^2 + \|f(\bm{\Lambda}_2)\|_F^2.
\end{equation}
Utilizing $q \geq 2$ we obtain
\begin{align*}
    &\|\bm{F}f(\bm{\Lambda}_1)\|_F \le \|\bm{\Lambda}_2^{q-3/2}\|_2 \|\bm{\Lambda}_1^{-(q-1/2)} f(\bm{\Lambda}_1)\|_2 \|\bm{\Lambda}_2\bm{\Omega}_2\bm{\Omega}_1^{\dagger}\|_F\\
    \leq&\gamma^{q-3/2}\|\bm{\Lambda}_1^{-1}f(\bm{\Lambda}_1)\|_2 \|\bm{\Lambda}_2\bm{\Omega}_2\bm{\Omega}_1^{\dagger}\|_F
    \le \gamma^{q-3/2} \frac{f(\lambda_k)}{\lambda_k} \|\bm{\Lambda}_2\bm{\Omega}_2\bm{\Omega}_1^{\dagger}\|_F,
\end{align*}
where the second inequality relies on $q \geq 2$ and the third inequality uses that $f(x)/x$ is decreasing (see Lemma~\ref{lemma:concave}). Plugging this inequality into~\eqref{eq:projection} yields
\begin{equation}\label{eq:first_term}
    \|(\bm{I}-\bm{P}_{\bm{Z}})f(\bm{\Lambda})\|_F^2 \leq \|f(\bm{\Lambda}_2)\|_F^2 + \gamma^{2(q-3/2)}\frac{f(\lambda_k)^2}{\lambda_k^2}\|\bm{\Lambda}_2 \bm{\Omega}_2 \bm{\Omega}_1^{\dagger}\|_F^2.
\end{equation}

For treating the second term in the sum~\eqref{eq:decomperror}, we decompose $f(\bm{\Lambda})$ into
\begin{equation*}
    f(\bm{\widetilde \Lambda}_1) = \begin{bmatrix} f(\bm{\Lambda}_1) & \\ & \bm{0} \end{bmatrix}, \quad f(\bm{\widetilde \Lambda}_2) = \begin{bmatrix} \bm{0} & \\ & f(\bm{\Lambda}_2) \end{bmatrix},
\end{equation*}
which gives
\begin{align*}
    \|(\bm{I}-\bm{P}_{\bm{Z}})f(\bm{\Lambda})\bm{P}_{\bm{Z}}\|_F & \leq
    \|(\bm{I}-\bm{P}_{\bm{Z}})f(\bm{\widetilde \Lambda}_1)\bm{P}_{\bm{Z}}\|_F+
    \|(\bm{I}-\bm{P}_{\bm{Z}})f(\bm{\widetilde \Lambda}_2)\bm{P}_{\bm{Z}}\|_F
    \\
    & \leq \|(\bm{I}-\bm{P}_{\bm{Z}})f(\bm{\widetilde \Lambda}_1)\|_F + \|f(\bm{\widetilde \Lambda}_2) \bm{P}_{\bm{Z}}\|_F.
\end{align*}
Replacing $f(\bm \Lambda)$ by $f(\bm{\widetilde \Lambda}_1)$ in~\eqref{eq:first_term} shows
$\|(\bm{I}-\bm{P}_{\bm{Z}})f(\bm{\widetilde \Lambda}_1)\|_F \leq \gamma^{q-3/2}\frac{f(\lambda_k)}{\lambda_k}\|\bm{\Lambda}_2\bm{\Omega}_2\bm{\Omega}_1^{\dagger}\|_F$. It turns out that the second term $\|f(\bm{\widetilde \Lambda}_2) \bm{P}_{\bm{Z}}\|_F = \|\bm{P}_{\bm{Z}}f(\bm{\widetilde \Lambda}_2)\|_F$ obeys the same bound:
\begin{align*}
    &\|\bm{P}_{\bm{Z}}f(\bm{\widetilde \Lambda}_2)\|_F^2 = \tr(f(\bm{\Lambda}_2) \bm{F}(\bm{I}+\bm{F}^T \bm{F})^{-1}\bm{F}^Tf(\bm{\Lambda}_2))\\
    \leq & \tr(f(\bm{\Lambda}_2)\bm{F}\bm{F}^T f(\bm{\Lambda}_2)) = \|f(\bm{\Lambda}_2)\bm{F}\|_F^2 \leq \gamma^{2(q-3/2)} \frac{f(\lambda_{k+1})^2}{\lambda_k^2} \|\bm{\Lambda}_2\bm{\Omega}_2\bm{\Omega}_1^{\dagger}\|_F^2\\
    \leq & \gamma^{2(q-3/2)} \frac{f(\lambda_{k})^2}{\lambda_k^2} \|\bm{\Lambda}_2\bm{\Omega}_2\bm{\Omega}_1^{\dagger}\|_F^2,
\end{align*}
where we used $(\bm{I} + \bm{F}^T \bm{F})^{-1} \preceq \bm{I}$ and the monotonicity of $f$. 
Overall one obtains
\begin{equation*}
    \|(\bm{I}-\bm{P}_{\bm{Z}})f(\bm{\Lambda})\bm{P}_{\bm{Z}}\|_F \leq 2 \gamma^{q-3/2}\frac{f(\lambda_k)}{\lambda_k}\|\bm{\Lambda}_2\bm{\Omega}_2\bm{\Omega}_1^{\dagger}\|_F.
\end{equation*}
Plugging this inequality and inequality~\eqref{eq:first_term} into~\eqref{eq:decomperror} completes the proof.
\end{proof}

\subsubsection{Probabilistic bounds}

Next, we proceed with turning Lemma~\ref{lemma:function_opmon_structural} into the probabilisitic bounds of Theorems~\ref{theorem:frobenius_opmon_expectation} and~\ref{theorem:frobenius_opmon_probability}. For this purpose we need the following results, which 
are common in the literature on randomized low-rank approximation; see, e.g.,~\cite[Sec. 10.1]{rsvd} and~\cite[Lemma 7]{gittensmahoney}.
\begin{lemma}\label{lemma:probabilistic}
Let $\bm{\Omega}_1 \in \mathbb{R}^{k \times (k+p)}$ and $\bm{\Omega}_2 \in \mathbb{R}^{(n-k) \times (k+p)}$ be independent Gaussian matrices. If $\bm{D}$ is a matrix and $k,p \geq 2$, then
\begin{equation}\label{eq:expectation}
    \mathbb{E}\|\bm{D}\bm{\Omega}_2\bm{\Omega}_1^{\dagger}\|_F^2 = \frac{k}{p-1}\|\bm{D}\|_F^2.
\end{equation}
Let $k, p>4$ and $u,t \geq 1$, then
\begin{equation}\label{eq:probability}
    \|\bm{D}\bm{\Omega}_2\bm{\Omega}_1^{\dagger}\|_F \leq \sqrt{\frac{3k}{p+1}} t \|\bm{D}\|_F + \frac{e\sqrt{k+p}}{p+1} tu \|\bm{D}\|_2
\end{equation}
holds with probability $\geq 1 - 2t^{-p} -e^{-u^2/2}$. 
\end{lemma}

\begin{proof}[Proof of Theorem~\ref{theorem:frobenius_opmon_expectation}]
By taking expectation on both sides of the structural bound of Lemma~\ref{lemma:function_opmon_structural} and using~\eqref{eq:expectation}, we obtain
\begin{equation*}
    \mathbb{E}\|f(\bm{A})-f(\widehat{\bm{A}}_{q,k+p})\|_F^2 \leq \|f(\bm{\Lambda}_2)\|_F^2 + 5\gamma^{2(q-3/2)}\frac{k}{p-1} \frac{f(\lambda_k)^2}{\lambda_k^2}\|\bm{\Lambda}_2\|_F^2. 
\end{equation*}
The proof is completed by using $\frac{f(\lambda_k)^2}{\lambda_k^2}\|\bm{\Lambda}_2\|_F^2 \leq \|f(\bm{\Lambda}_2)\|_F^2$ from Lemma~\ref{lemma:concave}.
\end{proof}

\begin{proof}[Proof of Theorem~\ref{theorem:frobenius_opmon_probability}]
Using the structural bound of Lemma~\ref{lemma:function_opmon_structural}, subadditivity of the square-root, and~\eqref{eq:probability} yields
\begin{align*}
    &\|f(\bm{A})-f(\widehat{\bm{A}}_{q,k+p})\|_F \\
    \leq & \|f(\bm{A})\|_F + \gamma^{q-3/2} \sqrt{\frac{15k}{p+1}}t \frac{f(\lambda_k)}{\lambda_k}\|\bm{\Lambda}_2\|_F + \gamma^{q-3/2}\frac{e\sqrt{5(k+p)}}{p+1}tu \frac{f(\lambda_k)}{\lambda_k}\|\bm{\Lambda}_2\|_2
\end{align*}
with probability at least $1 - 2t^{-p} - e^{-u^2/2}$.
The proof is completed by using $\frac{f(\lambda_k)^2}{\lambda_k^2}\|\bm{\Lambda}_2\|_{(s)}^2 \leq \|f(\bm{\Lambda}_2)\|_{(s)}^2$ from Lemma~\ref{lemma:concave} for $s = 2,\infty$, where $\|\cdot\|_{(s)}$ denotes the Schatten-$s$ norm.\end{proof}

\subsubsection{Improved bounds for the matrix square root}\label{section:frosqrtm}

Although the square root $f(x) = \sqrt{x}$ satisfies the conditions of Setting~\ref{setting} and the analysis above applies, it turns out that simpler and stronger bounds can be derived in this case. By Lemma~\ref{lemma:upper_bound} we have
\begin{equation*}
    \|\bm{A}^{1/2}-\widehat{\bm{A}}_{q,k+p}^{1/2}\|_F^2 \leq \|\bm{A}^{1/2}\|_F^2 - \|\widehat{\bm{A}}_{q,k+p}^{1/2}\|_F^2 = \tr(\bm{A} - \widehat{\bm{A}}_{q,k+p}).
\end{equation*}
Because of $\bm{A}-\widehat{\bm{A}}_{q,k+p} \succeq \bm{0}$ we have $\tr(\bm{A} - \widehat{\bm{A}}_{q,k+p}) = \|\bm{A}-\widehat{\bm{A}}_{q,k+p}\|_*$. This allows us to apply~\cite[Theorem 4]{gittensmahoney} and obtain that
\begin{equation*}
    \|\bm{A}^{1/2}-\widehat{\bm{A}}^{1/2}_{q,k+p}\|_F^2 \leq \|\bm{\Lambda}_2^{1/2}\|_F^2 + \gamma^{2(q-1)}\|\bm{\Lambda}_2^{1/2}\bm{\Omega}_2\bm{\Omega}_1^{\dagger}\|_F^2.
\end{equation*}
Compared to the structural bound of Lemma~\ref{lemma:function_opmon_structural}, this bound holds already for $q\geq 1$, it improves the exponent of $\gamma$ as well as the constants. Expectation and deviation bounds can be easily derived using Lemma~\ref{lemma:probabilistic}. 

\subsection{Nuclear norm error bounds}\label{section:nuclearnormerrorbounds}

In this section we state and prove the following probabilistic bounds for the approximation error~\eqref{eq:error} in the nuclear norm. 
\begin{theorem}\label{theorem:nuclear_opmon_expectation}
Under Setting~\ref{setting} suppose that $k, p \geq 2$ and $q \geq 1$. Then 
\begin{equation*}
    \mathbb{E}\|f(\bm{A})-f(\widehat{\bm{A}}_{q,k+p})\|_* \leq \left(1 +  \gamma^{2(q-1)} \frac{k}{p-1}\right) \|f(\bm{\Lambda}_2)\|_*.
\end{equation*}
\end{theorem}
\begin{theorem}\label{theorem:nuclear_opmon_probability}
    Under Setting~\ref{setting} suppose that $k,p > 4$ and $u,t,q \geq 1$. Then
    \begin{align*}
        &\|f(\bm{A})-f(\widehat{\bm{A}}_{q,k+p})\|_* \leq \left(1 + \gamma^{2(q-1)}\frac{3k}{p+1}\right)\|f(\bm{\Lambda}_2)\|_* +\\
        &\gamma^{2(q-1)}\left( \frac{e^2(k+p)}{(p+1)^2} t^2 u^2 \|f(\bm{\Lambda}_2)\|_2 + \frac{2e\sqrt{3k(k+p)}}{(p+1)^{3/2}} tu \sqrt{\|f(\bm{\Lambda}_2)\|_* \|f(\bm{\Lambda}_2)\|_2}\right).
    \end{align*}
    holds with probability at least $1-2t^{-p} - e^{-u^2/2}$.
    
\end{theorem}
Because of $f(\bm{A}) \succeq f(\widehat{\bm{A}}_{q,k+p})$ the nuclear norm of $f(\bm{A})-f(\widehat{\bm{A}}_{q,k+p})$ simplifies to the trace. Hence, Theorem~\ref{theorem:nuclear_opmon_expectation} and Theorem~\ref{theorem:nuclear_opmon_probability} provide bounds on the error of the trace approximation
\begin{equation*}
    \tr(f(\widehat{\bm{A}}_{q,k+p})) \approx \tr(f(\bm{A})).
\end{equation*} Similar bounds appear in \cite{saibaba} for $f(x) = \log(1+x)$; see Section~\ref{section:trest} for further discussion.

\subsubsection{Structural bound}

We now proceed with deriving a structural bound for the nuclear norm error. 
\begin{lemma}\label{lemma:nuclear_opmon_structural}
Under Setting~\ref{setting} and assuming $\rank(\bm{\Omega}_1) = k$ we have
\begin{equation*}
    \|f(\bm{A})-f(\widehat{\bm{A}}_{q,k+p})\|_* \leq \|f(\bm{\Lambda}_2)\|_* + \gamma^{2(q-1)} \frac{f(\lambda_k)}{\lambda_k} \|\bm{\Lambda}_2^{1/2}\bm{\Omega}_2 \bm{\Omega}_1^{\dagger}\|_F^2.
\end{equation*}
\end{lemma}
\begin{proof} As in the proof of Lemma~\ref{lemma:function_opmon_structural}, we set $\bm Y=\bm{A}^{q-1/2} \bm \Omega$.
Applying Lemma~\ref{lemma:move_out_projection} for the nuclear norm yields
\[
    \|f(\bm{A})-f(\widehat{\bm{A}}_{q,k+p})\|_* = \tr(f(\bm{A})) - \tr(f(\widehat{\bm{A}}_{q,k+p})) 
    \leq  \tr(f(\bm{A})) - \tr(\bm{P}_{\bm Y}f(\bm{A})\bm{P}_{\bm Y}).
\]
Using the factorization $\bm{P}_{\bm Y}f(\bm{A})\bm{P}_{\bm Y} = \bm{P}_{\bm Y}f(\bm{A})^{1/2} \big( \bm{P}_{\bm Y}f(\bm{A})^{1/2} \big)^T$, one obtains
\[
    \tr(f(\bm{A}) - \tr(\bm{P}_{\bm Y}f(\bm{A})\bm{P}_{\bm Y})
    = \|f(\bm{A})^{1/2}\|_F^2 - \|\bm{P}_{\bm Y}f(\bm{A})^{1/2}\|_F^2 = \|(\bm{I}-\bm{P}_{\bm Y})f(\bm{A})^{1/2}\|_F^2.
\]
Being a composition of operator monotone functions, the function $g = f^{1/2}$ is operator monotone~\cite[Exercise V.1.10]{bhatia}. Applying the inequality~\eqref{eq:projection} from the proof of Lemma~\ref{lemma:function_opmon_structural} to $g$ instead of $f$ yields
\begin{align*}
    \|(\bm{I}-\bm{P}_{\bm Y})g(\bm{A})\|_F^2  \leq & \|g(\bm{\Lambda}_2)\|_F^2 + \|\bm{F}g(\bm{\Lambda}_1)\|_F^2\\
    = & \|f(\bm{\Lambda}_2)\|_* + \|\bm{\Lambda}_2^{q-1/2} \bm{\Omega}_2\bm{\Omega}_1^{\dagger} \bm{\Lambda}_1^{-(q-1/2)}f(\bm{\Lambda}_1)^{1/2}\|_F^2\\
    \leq & \|f(\bm{\Lambda}_2)\|_* + \gamma^{2(q-1)}\|\bm{\Lambda}_1^{-1/2} f(\bm{\Lambda}_1)^{1/2}\|_2^2\|\bm{\Lambda}_2^{1/2} \bm{\Omega}_2\bm{\Omega}_1^{\dagger}\|_F^2\\
    = & \|f(\bm{\Lambda}_2)\|_* + \gamma^{2(q-1)}\frac{f(\lambda_k)}{\lambda_k}\|\bm{\Lambda}_2^{1/2} \bm{\Omega}_2\bm{\Omega}_1^{\dagger}\|_F^2,
\end{align*}
where the final equality follows from $\|\bm{\Lambda}_1^{-1/2} f(\bm{\Lambda}_1)^{1/2}\|_2^2 = f(\lambda_k)/\lambda_k$ because $f(x)/x$ is decreasing.
\end{proof}

\subsubsection{Probabilistic bounds}

As in the case of the Frobenius norm, we use the results of Lemma~\ref{lemma:probabilistic} to turn Lemma~\ref{lemma:nuclear_opmon_structural} into the probabilistic bounds of Theorems~\ref{theorem:nuclear_opmon_expectation} and~\ref{theorem:nuclear_opmon_probability}. 
\begin{proof}[Proof of Theorem~\ref{theorem:nuclear_opmon_expectation}]
By taking expectation on both sides of the inequality of Lemma~\ref{lemma:nuclear_opmon_structural} and using~\eqref{eq:expectation} we obtain
\begin{equation*}
    \mathbb{E}\|f(\bm{A})-f(\widehat{\bm{A}}_{q,k+p})\|_* \leq \|f(\bm{\Lambda}_2)\|_* + \gamma^{2(q-1)}\frac{k}{p-1} \frac{f(\lambda_k)}{\lambda_k}\|\bm{\Lambda}_2\|_*. 
\end{equation*}
Applying Lemma~\ref{lemma:concave} completes the proof.
\end{proof}

\begin{proof}[Proof of Theorem~\ref{theorem:nuclear_opmon_probability}]
Applying~\eqref{eq:probability} to the term 
$\|\bm{\Lambda}_2^{1/2}\bm{\Omega}_2 \bm{\Omega}_1^{\dagger}\|_F$ in the result of Lemma~\ref{lemma:nuclear_opmon_structural} and using $\|\bm{\Lambda}_2^{1/2}\|_F = \sqrt{\|\bm{\Lambda}_2\|_*}$, $\|\bm{\Lambda}_2^{1/2}\|_2 = \sqrt{\|\bm{\Lambda}_2\|_2}$ yield that
\begin{align*}
    &\|f(\bm{A})-f(\widehat{\bm{A}}_{q,k+p})\|_* \\
    \leq & \|f(\bm{A})\|_* + \gamma^{2(q-1)}\frac{f(\lambda_k)}{\lambda_k}\left(\sqrt{\frac{3k}{p+1}}t \sqrt{\|\bm{\Lambda}_2\|_*} + \frac{e\sqrt{k+p}}{p+1} tu \sqrt{\|\bm{\Lambda}_2\|_2}\right)^2
\end{align*}
holds with probability $\geq 1 - 2t^{-p} - e^{-u^2/2}$. Expanding the square and applying Lemma~\ref{lemma:concave} completes the proof.
\end{proof}

\subsubsection{Improved bounds for the matrix square-root}

As in the case for the Frobenius norm, it is possible to improve the exponent on $\gamma$ when $f(x) = \sqrt{x}$. By Lemma~\ref{lemma:nuclear_opmon_structural} we have
\begin{align*}
    \|\bm{A}^{1/2}-\widehat{\bm{A}}_{q,k+p}^{1/2}\|_* \leq & \|\bm{\Lambda}_2^{1/2}\|_* + \gamma^{2(q-1)} \frac{1}{\sqrt{\lambda_k}} \|\bm{\Lambda}_2^{1/2}\bm{\Omega}_2\bm{\Omega}_1^{\dagger}\|_F^2 \\
    \leq & \|\bm{\Lambda}_2^{1/2}\|_* + \gamma^{2q-3/2}\|\bm{\Lambda}_2^{1/4}\bm{\Omega}_2\bm{\Omega}_1^{\dagger}\|_F^2.
\end{align*}
Again, we obtain expectation and deviation bounds using \eqref{eq:expectation} and \eqref{eq:probability}. 

\subsection{Operator norm error bounds}

Finally, we present bounds for the error \eqref{eq:error} in the operator norm.
\begin{theorem}\label{theorem:operator_norm_expectation}
Under Setting~\ref{setting} suppose that $k,p \geq 2$ and $q \geq 1$. Then
\begin{align*}
    &\mathbb{E} \|f(\bm{A})-f(\widehat{\bm{A}}_{q,k+p}) \|_2 \\
    \leq & \|f(\bm{\Lambda}_2)\|_2 + \left\|f\left(\gamma^{2(q-1)}\frac{2k}{p-1}\bm{\Lambda}_2\right)\right\|_2 + \left\|f\left(\gamma^{2(q-1)} \frac{2e^2(k+p)}{p^2-1} \bm{\Lambda}_2\right)\right\|_*
\end{align*}
\end{theorem}
For conciseness we only state an expectation bound. From the proof of Theorem~\ref{theorem:operator_norm_expectation}, it follows that a deviation bound can be obtained from a deviation bound on the quantity $\|\bm{\Lambda}_2^{1/2}\bm{\Omega}_2\bm{\Omega}_1^{\dagger}\|_2$; see, e.g.,~\cite[Theorem 10.8]{rsvd} for such a bound. 


\subsubsection{Structural bound}

We first state a structural bound that holds in any unitarily invariant norm $\|\cdot\|$. 
The proof of this lemma is included in Appendix~\ref{appendix:operator_norm}.
\begin{lemma}\label{lemma:operator_norm_structural1}
Under Setting~\ref{setting}, assume that $\rank(\bm{\Omega}_1) = k$ and let $\bm{F} = \bm{\Lambda}_2^{q-1/2} \bm{\Omega}_2 \bm{\Omega}_1^{\dagger} \bm{\Lambda}_1^{-(q-1/2)}$.
Then
\begin{equation*}
    \|f(\bm{A})-f(\widehat{\bm{A}}_{q,k+p})\| \leq \|f(\bm{\Lambda}_2)\| + \|f(\bm{\Lambda}_1^{1/2} \bm{F}^T \bm{F} \bm{\Lambda}_1^{1/2})\|.
\end{equation*}
\end{lemma}

The result of Lemma~\ref{lemma:operator_norm_structural1} simplifies for the operator norm $\|\cdot\|_2$ because
\[
 \|f(\bm{\Lambda}_1^{1/2} \bm{F}^T \bm{F} \bm{\Lambda}_1^{1/2})\|_2 = 
 f(\|\bm{\Lambda}_1^{1/2} \bm{F}^T \bm{F} \bm{\Lambda}_1^{1/2}\|_2) \le 
 f\left(\gamma^{2(q-1)}\|\bm{\Lambda}_1^{1/2}\bm{\Omega}_2\bm{\Omega}_1^{\dagger}\|_2^2\right).
\]

\begin{corollary}\label{corollary:operator_norm_structural2}
Under Setting~\ref{setting} and assuming $\rank(\bm{\Omega}_1) = k$, one has
\begin{equation*}
    \|f(\bm{A})-f(\widehat{\bm{A}}_{q,k+p})\|_{2} \leq \|f(\bm{\Lambda}_2)\|_{2} + f\left(\gamma^{2(q-1)}\|\bm{\Lambda}_1^{1/2}\bm{\Omega}_2\bm{\Omega}_1^{\dagger}\|_2^2\right).
\end{equation*}
\end{corollary}

\subsubsection{Probabilistic bounds}

In order to obtain the expectation bound of Theorem~\ref{theorem:operator_norm_expectation} from Corollary~\ref{corollary:operator_norm_structural2} we make use of the following result established in~\cite[Appendix B]{frangella2021randomized}. 
\begin{lemma}\label{lemma:probabilistic3}
Let $\bm{\Omega}_1 \in \mathbb{R}^{k \times (k+p)}$ and $\bm{\Omega}_2 \in \mathbb{R}^{(n-k) \times (k+p)}$ be independent Gaussian matrices. If $\bm{D}$ is a matrix, then
\begin{equation}\label{eq:expectation3}
    \mathbb{E}\|\bm{D}\bm{\Omega}_2\bm{\Omega}_1^{\dagger}\|_2^2 \leq \frac{2k}{p-1}\|\bm{D}\|_2^2 + \frac{2e^2(k+p)}{p^2-1} \|\bm{D}\|_F^2
\end{equation}
\end{lemma}
\begin{proof}[Proof of Theorem~\ref{theorem:operator_norm_expectation}]
Using Corollary~\ref{corollary:operator_norm_structural2} and Jensen's inequality we obtain
\begin{align*}
    \mathbb{E}\|f(\bm{A})-f(\widehat{\bm{A}}_{q,k+p})\|_2 \leq & \|f(\bm{\Lambda}_2)\|_2 + \mathbb{E} f\left(\gamma^{2(q-1)}\|\bm{\Lambda}_2^{1/2}\bm{\Omega}_2\bm{\Omega}_1^{
    \dagger}\|_2^2\right)\\
    \leq & \|f(\bm{\Lambda}_2)\|_2 + f\left(\gamma^{2(q-1)}\mathbb{E}\|\bm{\Lambda}_2^{1/2}\bm{\Omega}_2\bm{\Omega}_1^{
    \dagger}\|_2^2\right).
\end{align*}
Bounding $\mathbb{E}\|\bm{\Lambda}_2^{1/2}\bm{\Omega}_2\bm{\Omega}_1^{\dagger}\|_2^2$ via \eqref{eq:expectation3} and using the subadditivity of $f$ on $[0,\infty)$ as well as the relations $\|\bm{\Lambda}_2^{1/2}\|_2^2 = \|\bm{\Lambda}_2\|_2$ and $\|\bm{\Lambda}_2^{1/2}\|_F^2 = \|\bm{\Lambda}_2\|_*$ we obtain
\begin{align*}
    &f\left(\gamma^{2(q-1)}\mathbb{E}\|\bm{\Lambda}_2^{1/2}\bm{\Omega}_2\bm{\Omega}_1^{
    \dagger}\|_2^2\right) \leq f\left(\gamma^{2(q-1)}\left(\frac{2k}{p-1}\|\bm{\Lambda}_2\|_2 + \frac{2e^2(k+p)}{p^2-1} \|\bm{\Lambda}_2\|_*\right)\right)\\
    \leq & f\left(\gamma^{2(q-1)}\frac{2k}{p-1}\|\bm{\Lambda}_2\|_2\right) + f\left(\gamma^{2(q-1)}\frac{2e^2(k+p)}{p^2-1} \|\bm{\Lambda}_2\|_*\right).
\end{align*}
Noting that $f\left(\gamma^{2(q-1)}\frac{2k}{p-1}\|\bm{\Lambda}_2\|_2\right) = \big\|f\left(\gamma^{2(q-1)}\frac{2k}{p-1}\bm{\Lambda}_2\right)\big\|_2$ and using once again the subadditivity of $f$ we have
\begin{equation*}
    f\left(\gamma^{2(q-1)}\frac{2e^2(k+p)}{p^2-1} \|\bm{\Lambda}_2\|_*\right) \leq \left\|f\left(\gamma^{2(q-1)} \frac{2e^2(k+p)}{p^2-1} \bm{\Lambda}_2\right)\right\|_*,
\end{equation*}
which completes the proof.
\end{proof}

\subsubsection{Improved bounds for the matrix square-root}

Once again improved results can be obtained in a relatively simple manner for $f(x) = \sqrt{x}$. Using Corollary~\ref{corollary:operator_norm_structural2} we have
\begin{equation*}
    \|\bm{A}^{1/2} - \widehat{\bm{A}}_{q,k+p}^{1/2}\|_2 \leq \|\bm{\Lambda}_2^{1/2}\|_2 + \gamma^{q-1} \|\bm{\Lambda}_2^{1/2}\bm{\Omega}_2\bm{\Omega}_2^{\dagger}\|_2
\end{equation*}
This can be turned into an expectation bound by using
\begin{equation*}
    \mathbb{E}\|\bm{\Lambda}_2^{1/2}\bm{\Omega}_2\bm{\Omega}_1^{\dagger}\|_2 \leq \sqrt{\frac{k}{p-1}} \|\bm{\Lambda}_2^{1/2}\|_2 + \frac{e\sqrt{k+p}}{p} \|\bm{\Lambda}_2^{1/2}\|_F,
\end{equation*}
which holds for $k,p \geq 2$; see, e.g., the proof of \cite[Theorem 10.6]{rsvd}. 
\section{Numerical Experiments}

In this section, we verify the performance of funNyström numerically. All experiments have been performed in MATLAB (version 2020a) on a MacBook Pro with a 2.3 GHz Intel Core i7 processor with 4 cores. Scripts to reproduce all figures in this paper are available at \url{https://github.com/davpersson/funNystrom}. 

We compare the approximation $f(\widehat{\bm{A}}_{q,k})$ returned by funNyström with the following references. Applying Nyström directly to $f(\bm{A})$ with an $n\times k$ Gaussian random matrix $\bm{\Omega}$ yields the rank-$k$ approximation
\begin{equation} 
    \widehat{\bm{B}}_{q,k} = f(\bm{A})^q \bm{\Omega}(\bm{\Omega}^Tf(\bm{A})^{2q-1}\bm{\Omega})^{\dagger}(f(\bm{A})^q \bm{\Omega})^T. \label{eq:nystromf}
\end{equation}
This assumes that mvps with $f(\bm{A})$ are carried out \emph{exactly}. If each mvp with $f(\bm{A})$ needed in~\eqref{eq:nystromf} is approximated using $d$ iterations of the Lanczos method, one obtains a different approximation, which will be denoted by $\widehat{\bm{B}}^{(d)}_{q,k}$.

Given an approximation $\bm B$ of $f(\bm{A})$, we will measure the relative error
$\|f(\bm{A})-\bm{B}\| / \|f(\bm{A})\|$ for some norm $\|\cdot\|$. 

\subsection{Test matrices}
In the following, we describe the test matrices used in our experiments. 

\subsubsection{Synthetic matrices}

We consider synthetic matrices with prescribed algebraic and exponential eigenvalue decays. Let $\bm{\Lambda}_{\text{alg}}$ and $\bm{\Lambda}_{\text{exp}}$ be diagonal matrices with diagonal entries
\[
    (\bm{\Lambda}_{\text{alg}})_{ii} = si^{-c}, \quad 
    (\bm{\Lambda}_{\text{exp}})_{ii} = s \gamma^{i}, \quad
    i = 1,\ldots,n,
\]
for parameters $s,c > 0$ and $\gamma \in (0,1)$.
Letting $\bm{U}$ denote the orthogonal matrix generated by the MATLAB-command $\texttt{gallery('orthog',n,1)}$, we set
\begin{equation}
    \bm{A}_{\text{alg}} = \bm{U} \bm{\Lambda}_{\text{alg}}\bm{U}^T, \quad 
    \bm{A}_{\text{exp}} = \bm{U} \bm{\Lambda}_{\text{exp}}\bm{U}^T. \label{eq:exponential_decay}
\end{equation}
Unless specified otherwise, we choose $n = 5000$. 

\subsubsection{Gaussian process covariance kernels}

We consider two classes of matrices that arise from the discretization of the squared exponential and Matérn Gaussian process covariance kernels \cite{seeger2004gaussian}. For this purpose, we generate $n = 5000$ i.i.d. data points $x_1,\ldots,x_{5000} \sim N(0,1)$ and set
\begin{align}
    &\bm{A}_{\text{SE}} \in \mathbb{R}^{n \times n}, \quad (\bm{A}_{\text{SE}})_{ij} = \exp\left(-|x_i-x_j|^2/(2\sigma^2)\right), \label{eq:squared_exponential}\\
    &\textbf{A}_{\text{Mat}} \in \mathbb{R}^{n \times n}, \quad (\textbf{A}_{\text{Mat}})_{ij} = \frac{\pi^{1/2}\left(\alpha |x_i - x_j|\right)^{\nu} K_{\nu}(\alpha|x_i-x_j|)}{2^{\nu - 1}\Gamma(\nu + 1/2)) \alpha^{2\nu}}, \label{eq:matern}
\end{align}
for $i,j = 1,\ldots,n$, and parameters $\sigma, \alpha, \nu > 0$. Note that $K_{\nu}$ is the modified Bessel function of the second kind. Computing $\tr\left(\log(\bm{I} + \bm{A})\right)$, where $\bm{A}$ is a matrix arising from discretizing a covariance kernel, is an important task in Bayesian optimization and maximum likelihood estimation for Gaussian processes \cite{gardner2018gpytorch, wenger2021reducing}. 

\subsubsection{Bayesian inverse problem}

Motivated by the numerical experiments in \cite{alexanderian2014optimal,dudley2020monte,saibaba}, this test matrix arises from a Bayesian inverse problem. Consider the parabolic partial differential equation
\begin{align}
\begin{split}
    &u_t = \kappa \Delta u + \lambda u \text{ in } [0,1]^2 \times [0,2]\\
    &u(\cdot,0) = \theta \text{ in } \mathcal{D}\\
    &u = 0 \text{ on } \Gamma_1\\
    & \frac{\partial u}{\partial \bm{n}} = 0 \text{ on } \Gamma_2
\end{split}
\label{eq:pde}
\end{align}
for $\kappa, \lambda > 0$ and $\Gamma_2 = \{(x,1) \in \mathbb{R}^{2} : x \in [0,1] \}$ and $\Gamma_1 = \partial \mathcal{D} \setminus \Gamma_2$. We place 49 sensors  at $(i/8,j/8) \in [0,1]^2$ for $i,j = 1,\ldots,7$ to take measurements of $u$ at these sensor locations at times $t = 1, 1.5, 2$. We gather all $3\times 49 = 147$ measurements in a vector $\bm{d} \in \mathbb{R}^{147}$.

Discretizing~\eqref{eq:pde} in space using finite differences on $40 \times 40$ equispaced grid yields an ordinary differential equation of the form
\begin{align}
\begin{split}
    \dot{\bm{u}}(t) &= \bm{L} \bm{u}(t) \text{ for } t \in [0,2],\\
    \bm{u}(0) &= \bm{\theta}. \label{eq:ode}
\end{split}
\end{align}
The solution to \eqref{eq:ode} is $\bm{u}(t) = \exp(t\bm{L}) \bm{\theta}$. Let $\bm{u}_{\text{measure}}\in \mathbb{R}^{147}$ contain the values of $\bm{u}$ corresponding to sensor locations at times $t = 1, 1.5, 2$. Then, by linearity, we can write $\bm{u}_{\text{measure}} = \bm{F}\bm{\theta}$ for a matrix $\bm{F}$.  

Assume that $\bm{\theta} \sim N(\bm{\theta}_0,\bm{\Sigma}_0)$, the discretization error is negligible, and that the measurements $\bm{d}$ are distorted by some noise $\bm{\varepsilon} \sim N(\bm{0}, \bm{\Sigma}_{\text{noise}})$ so that
\begin{equation*}
    \bm{d} = \bm{u}_{\text{measure}} + \bm{\varepsilon}.
\end{equation*}
It is well known that the posterior distribution of $\bm{\theta}$ is given by
$\bm{\theta}|\bm{d} \sim N(\bm{\theta}_{\text{post}},\bm{\Sigma}_{\text{post}})$ with
\[
\bm{\theta}_{\text{post}} = \bm{\Sigma}_{\text{post}}(\bm{F}^T \bm{\Sigma}_{\text{noise}}^{-1}\bm{d} + \bm{\Sigma}_0^{-1} \bm{\theta}_0),\quad 
\bm{\Sigma}_{\text{post}} = (\bm{F}^T \bm{\Sigma}_{\text{noise}}^{-1}\bm{F} + \bm{\Sigma}_0^{-1})^{-1};
\]
see \cite{stuart}. Now let
\begin{equation}\label{eq:h0}
    \bm{A}_{\text{pde}} = \bm{\Sigma}_0^{1/2}\bm{F}^T \bm{\Sigma}_{\text{noise}}^{-1}\bm{F}\bm{\Sigma}_0^{1/2}.
\end{equation}
Then, $\tr(\log(\bm{I} + \bm{A}_{\text{pde}}))$ is related to the expected information gain from the posterior distribution relative to the prior distribution~\cite{alexanderian2014optimal}. For fine discretization grids, the matrix $\bm{F}$, and thus $\bm{A}_{\text{pde}}$, cannot be formed explicitly. Instead, one only implicitly performs mvps with $\bm{A}_{\text{pde}}$ via solving \eqref{eq:ode}. 

\subsection{Comparing number of mvps}\label{section:numerical_experiments_accuracy}

Recall that Algorithm~\ref{alg:funnystrom} requires $qk$ mvps with $\bm{A}$ to compute $f(\widehat{\bm{A}}_{q,k})$.
In contrast, the approximation $\widehat{\bm{B}}^{(d)}_{q,k}$ -- obtained via applying Nystr\"om to $f(\bm{A})$ --
requires $d q k$ mvps with $\bm{A}$. 
The choice of $d$, the number of Lanczos iterations, needs to be chosen in dependence of $q,k$ such that the impact on the overall accuracy remains negligible. For the purpose of our numerical comparison, we have precomputed the matrix $\widehat{\bm{B}}_{q,k}$ obtained without the additional Lanczos approximation and choose $d$ such that
\begin{equation}
    \|f(\bm{A})-\widehat{\bm{B}}^{(d)}_{q,k}\| \leq 1.1\cdot  \|f(\bm{A})-\widehat{\bm{B}}_{q,k}\|.\label{eq:condition}
\end{equation}
In practice, $\widehat{\bm{B}}_{q,k}$ is not available and one needs to employ heuristic and potentially less reliable criteria.
In our implementation we increase $d$ by 5 until \eqref{eq:condition} is satisfied. The results obtained for  $q = 1$ are reported in Figure~\ref{fig:savings}. Clearly, Algorithm~\ref{alg:funnystrom} needs fewer mvps; the difference can be up to three orders of magnitude.
\begin{figure}
\begin{subfigure}{.5\textwidth}
  \centering
  \includegraphics[width=\linewidth]{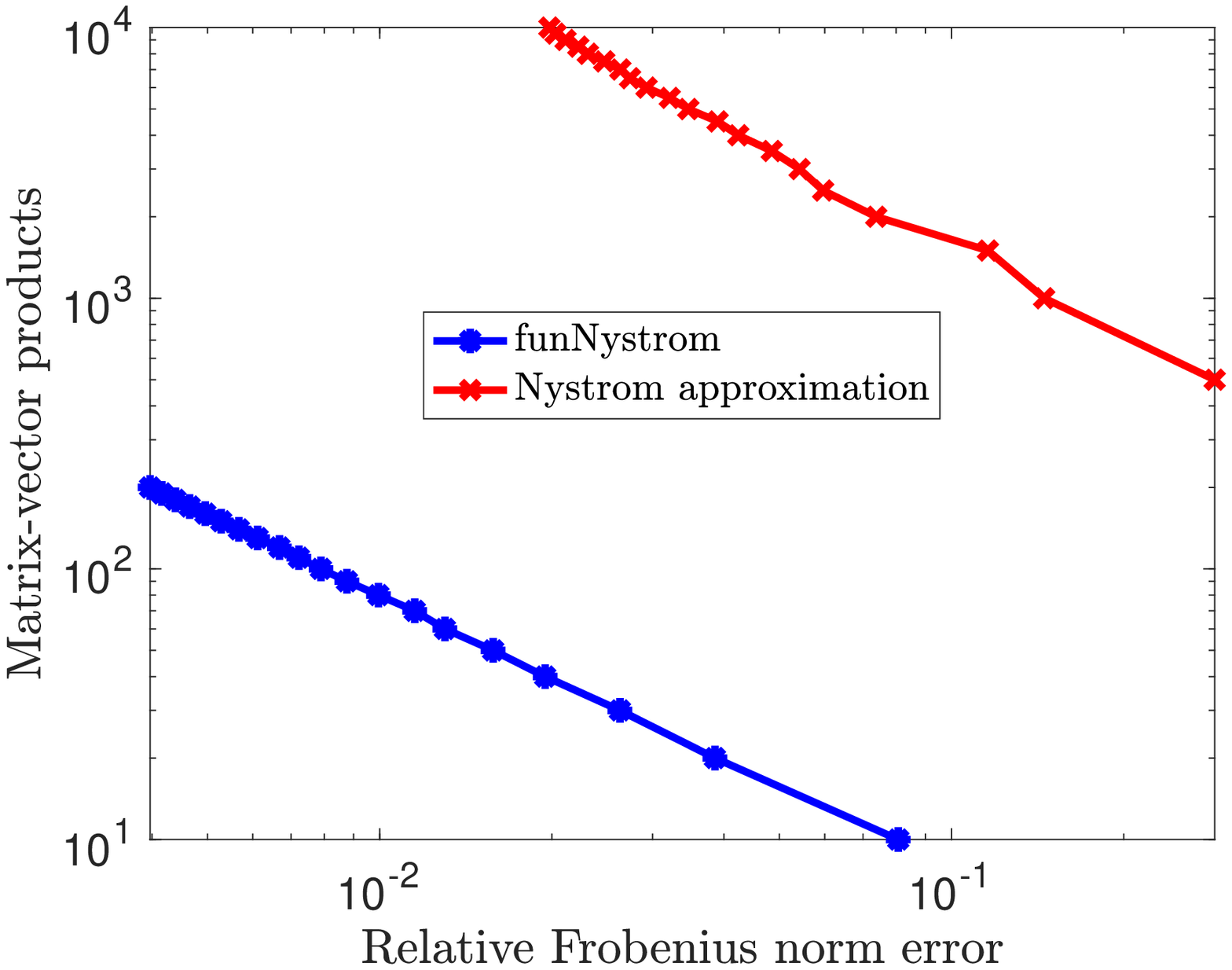}
    \caption{$\bm{A}_{\text{alg}}$ defined in \eqref{eq:exponential_decay} with $s=1$, $c = 3$ and $f(x) = x^{1/2}$.}
\end{subfigure}
\begin{subfigure}{.5\textwidth}
  \centering
  \includegraphics[width=\linewidth]{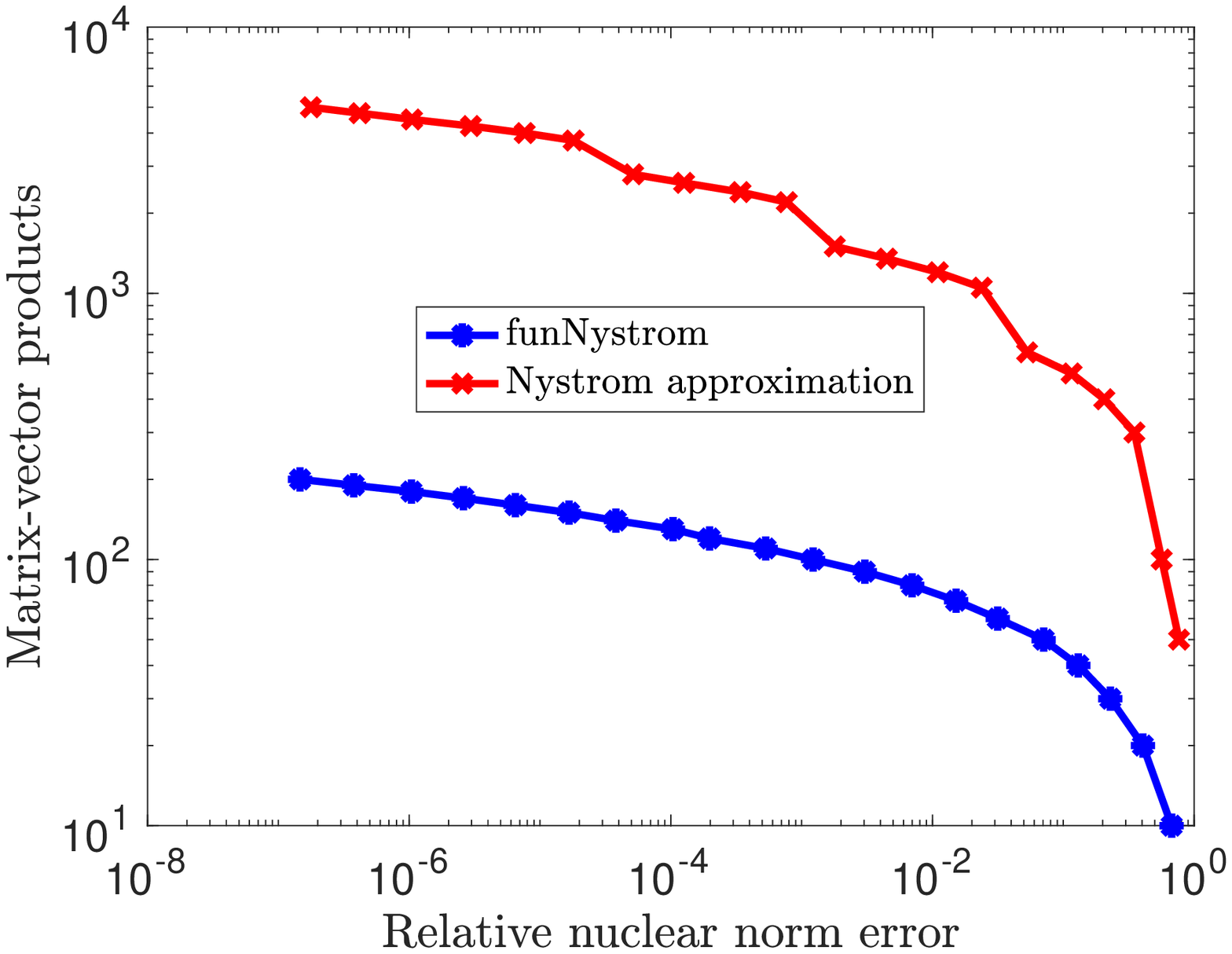}  
\caption{$\bm{A}_{\text{exp}}$ defined in \eqref{eq:exponential_decay} with $s = 10, \gamma = e^{-1/10}$ and $f(x) = \frac{x}{x+1}$.}
\end{subfigure}
\begin{subfigure}{.5\textwidth}
  \centering
  \includegraphics[width=\linewidth]{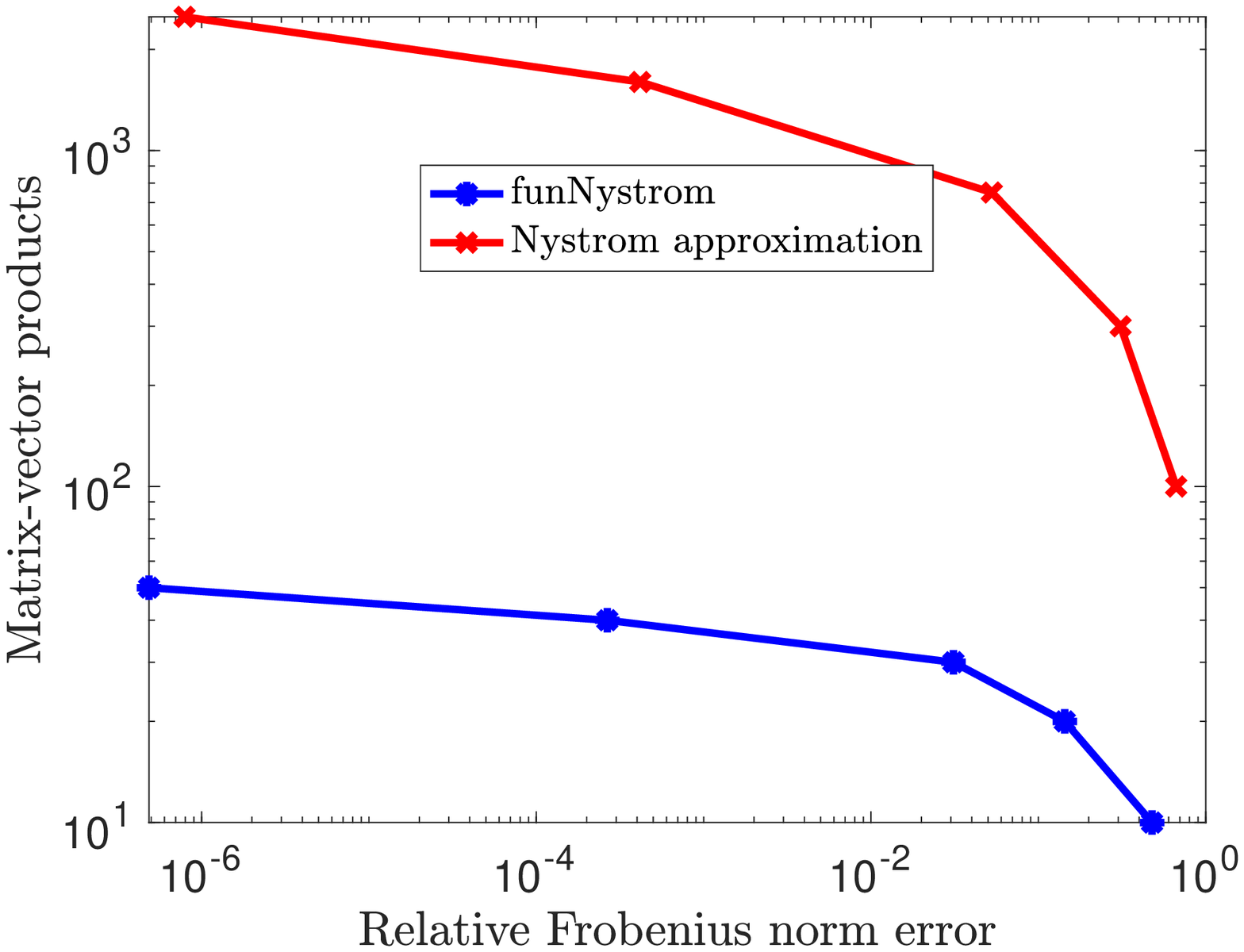}  
  \caption{$\bm{A}_{\text{SE}}$ defined in \eqref{eq:squared_exponential} with $\sigma^2 = 0.1$ and $f(x) = \log(1+x)$.}
\end{subfigure}
\begin{subfigure}{.5\textwidth}
  \centering
  \includegraphics[width=\linewidth]{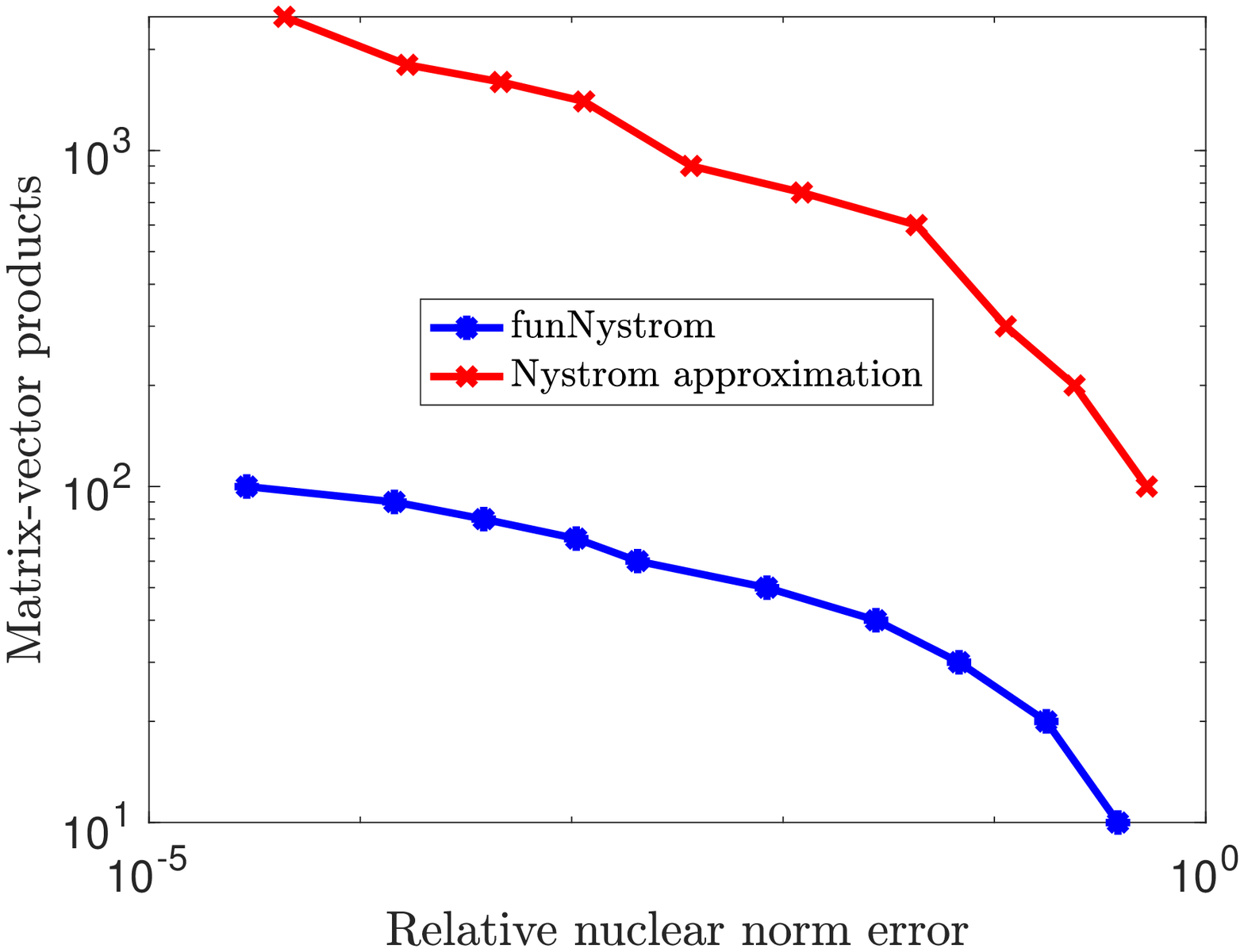}  
  \caption{$\bm{A}_{\text{pde}}$ defined in \eqref{eq:h0} with $\kappa = 0.01, \lambda = 1$, $\bm{\Sigma}_{\text{noise}} = \bm{I}$ and $f(x) = \log(1+x)$. }
\end{subfigure}
\caption{Number of mvps with $\bm{A}$ vs. attained accuracy for low-rank approximations of $f(\bm{A})$ from Algorithm~\ref{alg:funnystrom} with $q=1$ (blue) and $\widehat{\bm{B}}^{(d)}_{q,k}$ (red).}
\label{fig:savings}
\end{figure}%

\subsection{Comparing accuracy}

In Figure~\ref{fig:exact}, we compare the approximation error of Algorithm~\ref{alg:funnystrom} with the (significantly more expensive) approximation $\widehat{\bm{B}}_{q,k}$. It can be observed  that Algorithm~\ref{alg:funnystrom} is never worse than $\widehat{\bm{B}}_{q,k}$, and sometimes even better.
This suggests that even when mvps with $f(\bm{A})$ can be performed very efficiently, Algorithm~\ref{alg:funnystrom} may still be the preferred choice. In the figures we also plot the expectation error bounds from Section~\ref{section:frosqrtm} and Theorem~\ref{theorem:nuclear_opmon_expectation} for choices of $k$ and $p$ that minimize the right-hand side of the error bound. For $q = 1$ our bounds recover the empirical error up to a factor less than 10. For $q = 2$ the performance of Algorithm~\ref{alg:funnystrom} is sometimes significantly better than our bounds predict, suggesting that there is room to tighten the bounds for $q \geq 2$.

\begin{figure}
\begin{subfigure}{.5\textwidth}
  \centering
  \includegraphics[width=\linewidth]{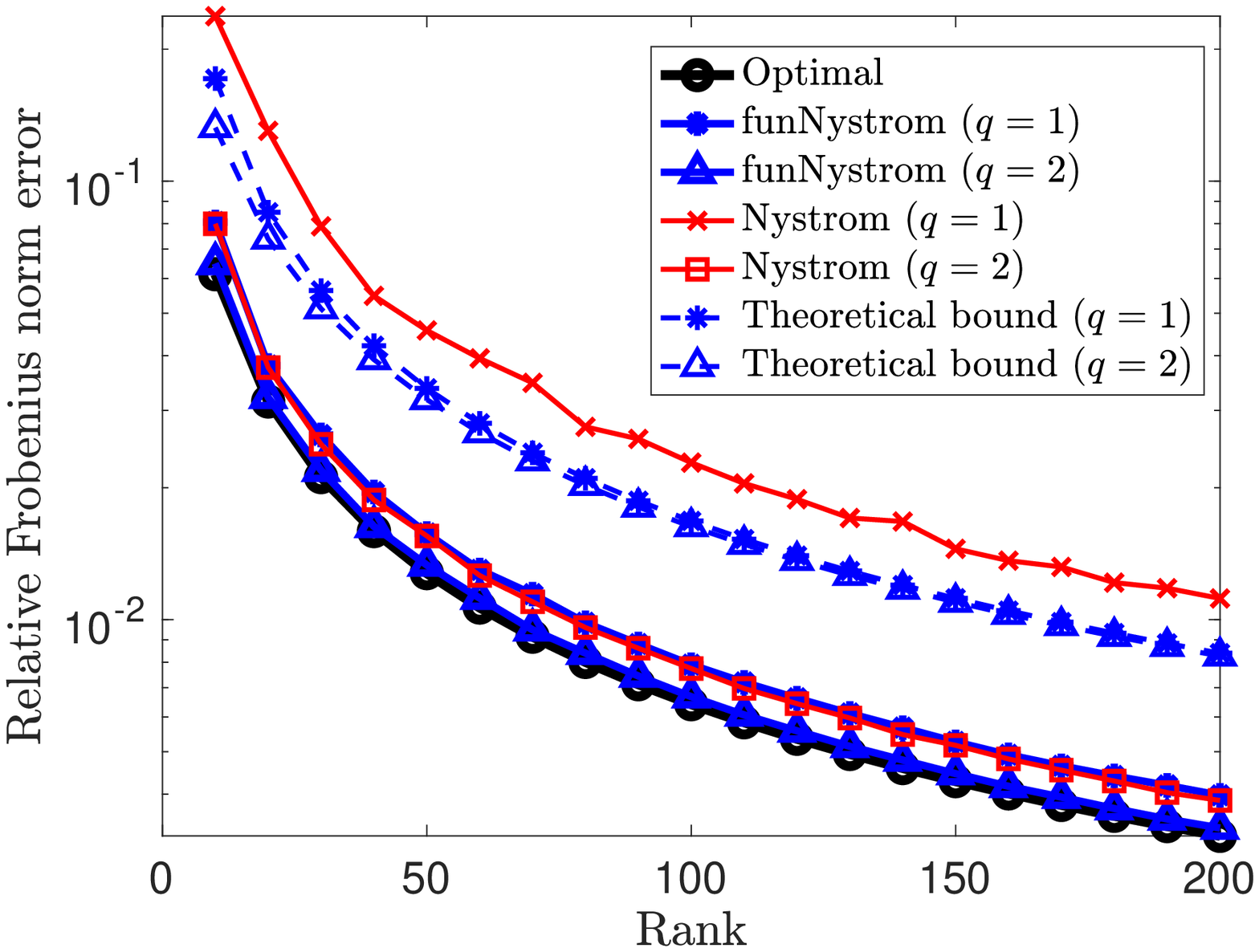}  
  \caption{$\bm{A}_{\text{alg}}$ defined in \eqref{eq:exponential_decay} with $s = 1$, $c = 3$ and $f(x) = x^{1/2}$.}
\end{subfigure}
\begin{subfigure}{.5\textwidth}
  \centering
  \includegraphics[width=\linewidth]{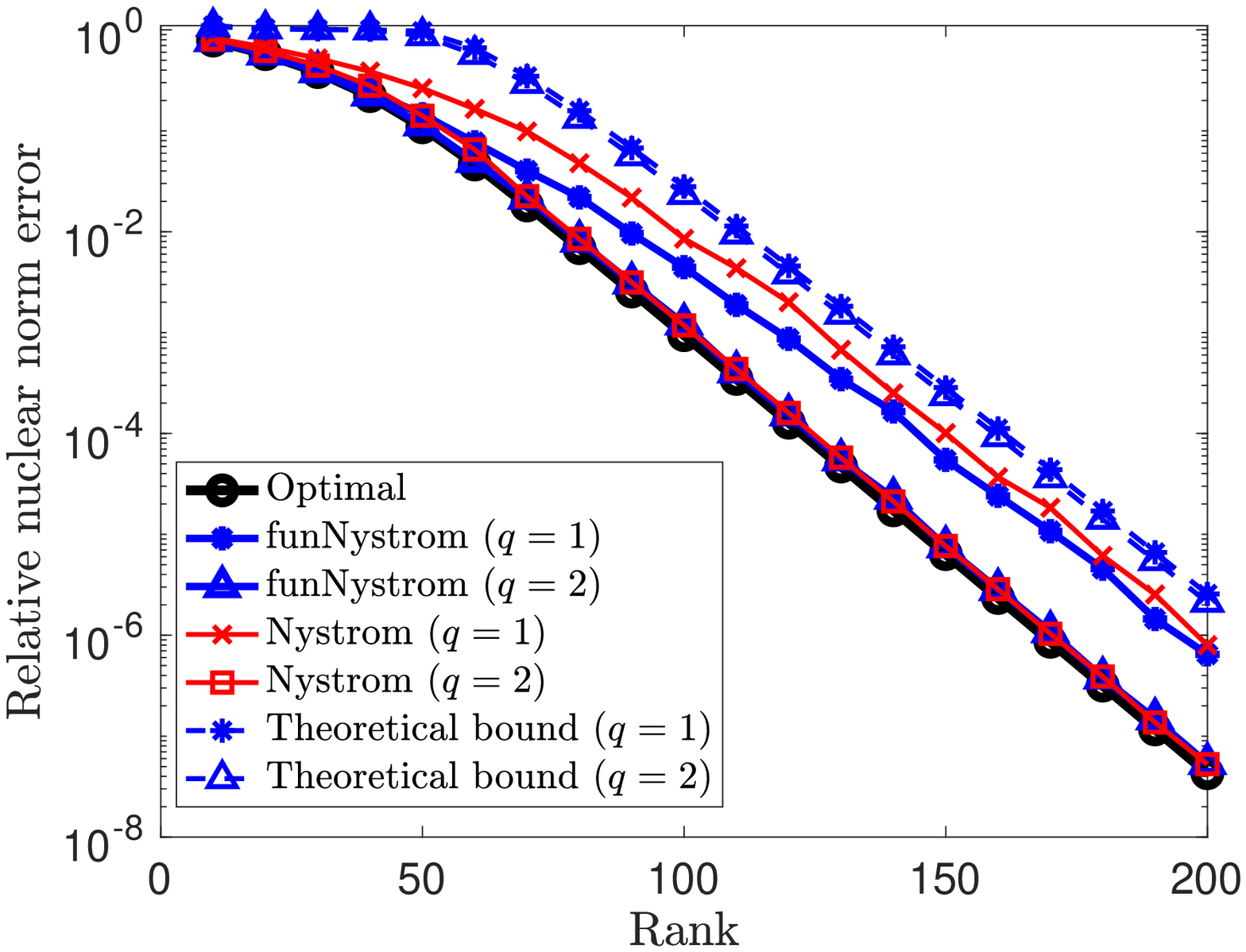}  
  \caption{$\bm{A}_{\text{exp}}$ defined in \eqref{eq:exponential_decay} with $s = 1$, $\gamma = e^{-\frac{1}{10}}$ and $f(x) = \frac{x}{x+0.01}$.}
\end{subfigure}
\begin{subfigure}{.5\textwidth}
  \centering
  \includegraphics[width=\linewidth]{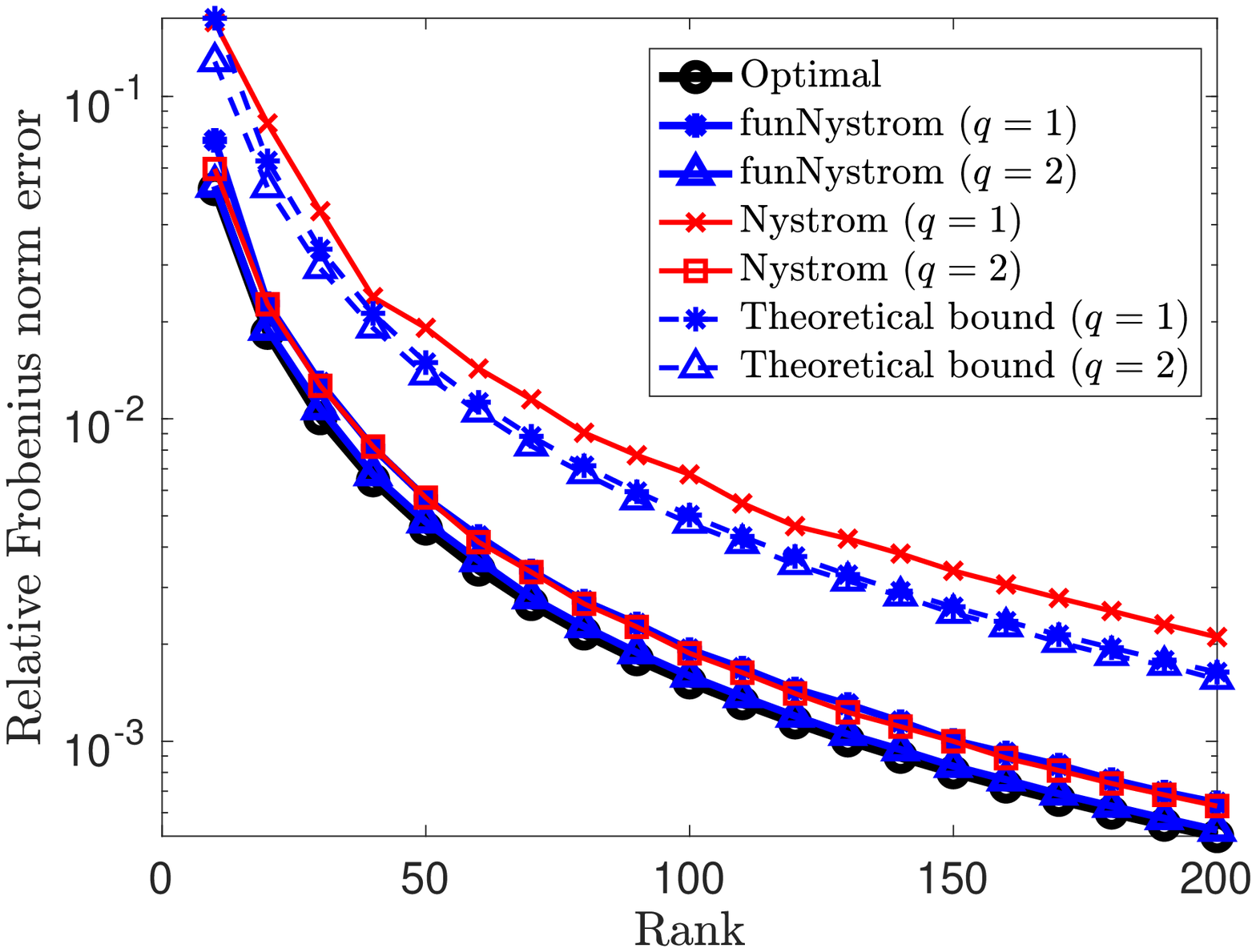}  
  \caption{$\bm{A}_{\text{Mat}}$ defined in \eqref{eq:matern} with $\alpha = 1$, $\nu = 3/2$ and $f(x) = x^{1/2}$.}
\end{subfigure}
\begin{subfigure}{.5\textwidth}
  \centering
  \includegraphics[width=\linewidth]{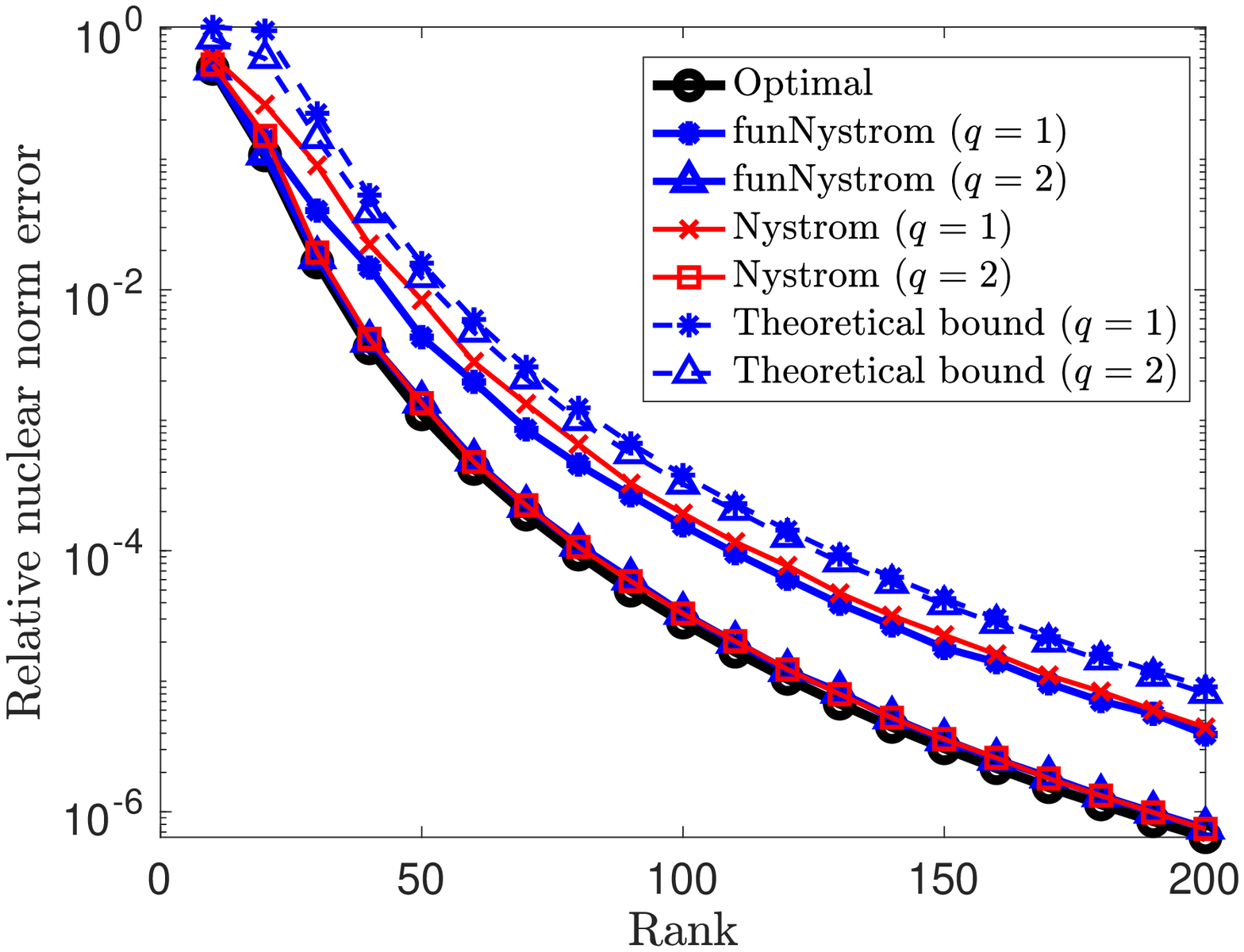}  
  \caption{$\bm{A}_{\text{Mat}}$ defined in \eqref{eq:matern} with $\alpha = 1$, $\nu = 5/2$ and $f(x) = \frac{x}{x+0.01}$.}
\end{subfigure}
\caption{Error vs. prescribed rank of the approximation for Algorithm~\ref{alg:funnystrom} applied to $\bm{A}$ (blue) and $\widehat{\bm{B}}_{q,k}$, the Nyström approximation applied to $f(\bm{A})$ (red), for $q =1,2$. }
\label{fig:exact}
\end{figure}

\subsection{Fast computation of mvps}


In this section we show that Algorithm~\ref{alg:funnystrom} can be used to compute fast mvps with $f(\bm{A})$. We let $f(x) = x^{1/2}$ and $\bm{A} = \bm{A}_{\text{exp}}$ defined in \eqref{eq:exponential_decay} with $s = 1, \gamma = e^{-1}$ and $n = 10000$. We let $\bm{Z} \in \mathbb R^{n \times N}$ be the matrix containing the first $N$ columns of the identity matrix. Hence, computing $\bm{A}^{1/2}\bm{Z}$ requires $N$ mvps with $\bm{A}^{1/2}$. We compare the computation times of the following two methods for approximating $\bm{A}^{1/2}\bm{Z}$:
\begin{enumerate}
    \item Approximating $\bm{A}^{1/2}\bm{Z}$ using the Lanczos method with $p$ iterations. This comes at a computational cost of $O(pn^2 N)$. The implementation we use for the Lanczos method is the same implementation used for the numerical expriments in \cite{hutchpp}, which approximates the $N$ mvps with $\bm{A}^{1/2}$ simultaneously by vectorizing all computations, rather than approximating the $N$ mvps subsequently. This significantly speeds up the computation.
    \item Computing $\widehat{\bm{A}}_{q,k}^{1/2}$ using Algorithm~\ref{alg:funnystrom} and approximate $\widehat{\bm{A}}^{1/2}_{q,k} \bm{Z} \approx \bm{A}^{1/2} \bm{Z}$. This comes at a computational cost of $O(qkn^2 + nkN)$.
\end{enumerate}
If we let $T_{\text{Lanczos}}(N)$ and $T_{\text{low-rank}}(N)$ be wall-clock time to approximate $N$ mvps with $\bm{A}^{1/2}$ using the Lanczos method and low-rank approximation respectively, the speed-up factor will be 
\begin{equation*}
    \frac{T_{\text{Lanczos}}(N)}{T_{\text{low-rank}}(N)} = O\left(\frac{pn^2N}{qkn^2 + nkN}\right) = O(N),
\end{equation*}
if we keep $p,q$ and $k$ constant and assume $N \ll n$.

In our numerical experiments we set $k = 14, q = 1, p = 21$. This choice of parameters yields a similar relative error $\|\bm{A}^{1/2}\bm{Z}-\bm{Y}\|_F/\|\bm{A}^{1/2}\bm{Z}\|_F \approx 0.01$ for both methods and for all $N$, where $\bm{Y}$ is the approximation to $\bm{A}^{1/2}\bm{Z}$. We set $N = 10, 20, \ldots, 100$. The results are presented in Figure~\ref{fig:comptimes}, which confirm the $O(N)$ speed-up factor.
\begin{figure}
    \centering
    \includegraphics[width=\linewidth]{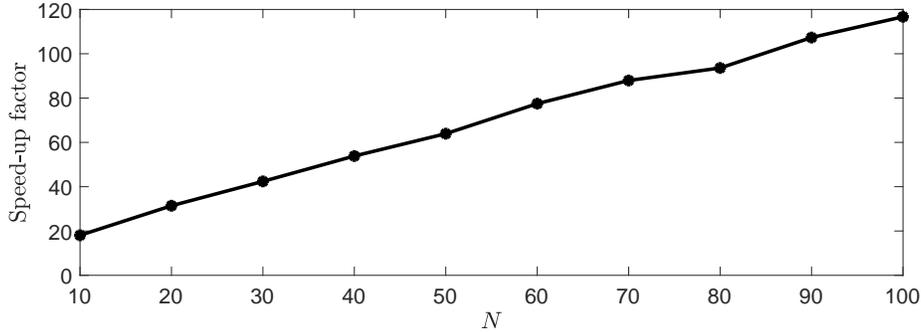}
    \caption{Number of mvps $N$ vs. speed-up factor $T_{\text{Lanczos}}(N)/T_{\text{low-rank}}(N)$.}
    \label{fig:comptimes}
\end{figure}

\section{Application to trace estimation}\label{section:trest}

In this section, we discuss how funNystr\"om can be used to approximate $\tr(f(\bm A))$, the trace of $f(\bm A)$, under Setting~\ref{setting}.

\subsection{Trace estimation via low-rank approximation}
When an $n\times n$ matrix $\bm{B}$ admits an excellent rank-$k$ approximation $\bm{B}_k$ for $k\ll n$, it is sensible to approximate $\tr(\bm{B})$ by $\tr(\bm{B}_k)$. Setting $\bm{B} = f(\bm{A})$, this motivates the approximation
\begin{equation} \label{eq:ourtrace}
 \tr(f(\bm{A}))\approx \tr(f(\widehat{\bm{A}}_{q,k+p})) = f(\hat \lambda_{1}) + \cdots + f(\hat \lambda_{k+p}),
\end{equation}
where $\widehat{\bm{A}}_{q,k+p} = \widehat{\bm{U}} \widehat{\bm{\Lambda}} \widehat{\bm{U}}^T$ denotes the output of Algorithm~\ref{alg:funnystrom} with $\widehat{\bm{\Lambda}} = \text{diag}(\hat \lambda_1,\ldots,\hat \lambda_{k+p})$. Using that $f(\bm{A}) \succeq f(\widehat{\bm{A}}_{q,k+p})$ we get
\[
    \tr(f(\bm{A}))-\tr(f(\widehat{\bm{A}}_{q,k+p})) = \|f(\bm{A})-f(\widehat{\bm{A}}_{q,k+p})\|_*.
\]
Hence, Theorem~\ref{theorem:nuclear_opmon_expectation} and Theorem~\ref{theorem:nuclear_opmon_probability} yield probabilistic bounds for the error of this trace approximation. 

It is instructive to compare our results with the bounds from~\cite{saibaba} for the special case $f(x) = \log(1+x)$. In particular, Theorem 1 from~\cite{saibaba}
states that
\begin{align}\label{eq:saibaba_expectationbound}
\begin{split}
    &\mathbb{E}\left[\tr\left(\log(\bm{I}+\bm{A})\right)-\tr\left(\log(\bm{I}+\bm{Q}^T \bm{A} \bm{Q})\right)\right] \\
    \leq & \tr\left(\log(\bm{I} + \bm{\Lambda}_2)\right) + \tr\left(\log(\bm{I} + \gamma^{2q-1}C \bm{\Lambda}_2)\right),
\end{split}
\end{align}
where $\bm{Q}$ is an orthonormal basis for $\range(\bm{A}^q\bm{\Omega})$, and \begin{equation*}
    C = \frac{e^2(k+p)}{(p+1)(p-1)} \left(\frac{1}{2\pi(p+1)}\right)^{\frac{2}{p+1}}(\sqrt{n-k} + \sqrt{k+p} + \sqrt{2})^2.
\end{equation*}
Constructing $\bm{Q}^T \bm{A} \bm{Q}$ requires a total of $(q+1)(k+p)$ mvps with $\bm{A}$. On the other hand, within the same budget one obtains the more accurate low-rank approximation $\widehat{\bm{A}}_{q+1,k+p}$. 
This also translates into tighter probabilistic bounds for trace estimation. To see this, note that Theorem~\ref{theorem:nuclear_opmon_expectation} gives
\[
     \mathbb{E}\left[\tr\big(\log(\bm{I}+\bm{A})\big)-\tr\big(\log(\bm{I}+\widehat{\bm{A}}_{q+1,k+p})\big) \right] \leq \left(1 +   \frac{\gamma^{2q} k}{p-1}\right) \tr\left(\log(\bm{I} + \bm{\Lambda}_2)\right).
\]
The difference between~\eqref{eq:saibaba_expectationbound} and this bound satisfies
\begin{equation*}
    \sum\limits_{i=k+1}^n \left[ \log(1+\gamma^{2q-1} C \lambda_i) - \frac{\gamma^{2q}k \log(1+\lambda_i)}{p-1} \right]
    \approx \gamma^{2q-1} \sum\limits_{i=k+1}^n \Big(C- \frac{\gamma k}{p-1} \Big)\lambda_i
\end{equation*}
for $\lambda_{k+1}\approx 0$.\footnote{We use $\log(1+x)\approx x$ for small $x$.} Because $C \ge \frac{0.55 kn}{(p+1)(p-1)}$ and usually $n\gg p$,
this shows that we obtain a much tighter bound for our method compared to~\cite{saibaba}.
Similarly, it can be shown that our deviation bounds are tighter than those in \cite{saibaba}. Similar bounds for $f(x) = \frac{x}{x+1}$ exist in \cite[Theorem A.1]{saibabarandom}. By an identical argument we can show that our bounds are tighter than those in \cite{saibabarandom}.

In Figure~\ref{fig:trace_est} we compare our approach, funNystrom combined with~\eqref{eq:ourtrace}, with the method presented in \cite{saibaba} to approximate $\tr(\log(\bm{I} + \bm{A}))$. We choose $q = 1$ for both methods, since we have observed that increasing the rank of the low-rank approximation often yields a more accurate low-rank approximation than increasing $q$. A budget of $m$ mvps allows one to choose $k+p = m$ in Algorithm~\ref{alg:funnystrom} while one can only choose $k+p = m/2$ in the method from~\cite{saibaba}. This explains the better performance of funNystrom observed in Figure~\ref{fig:trace_est}; a similar observation has been made in \cite[Section 3]{AHutchpp}.

\begin{figure}
\begin{subfigure}{.5\textwidth}
  \centering
  \includegraphics[width=\linewidth]{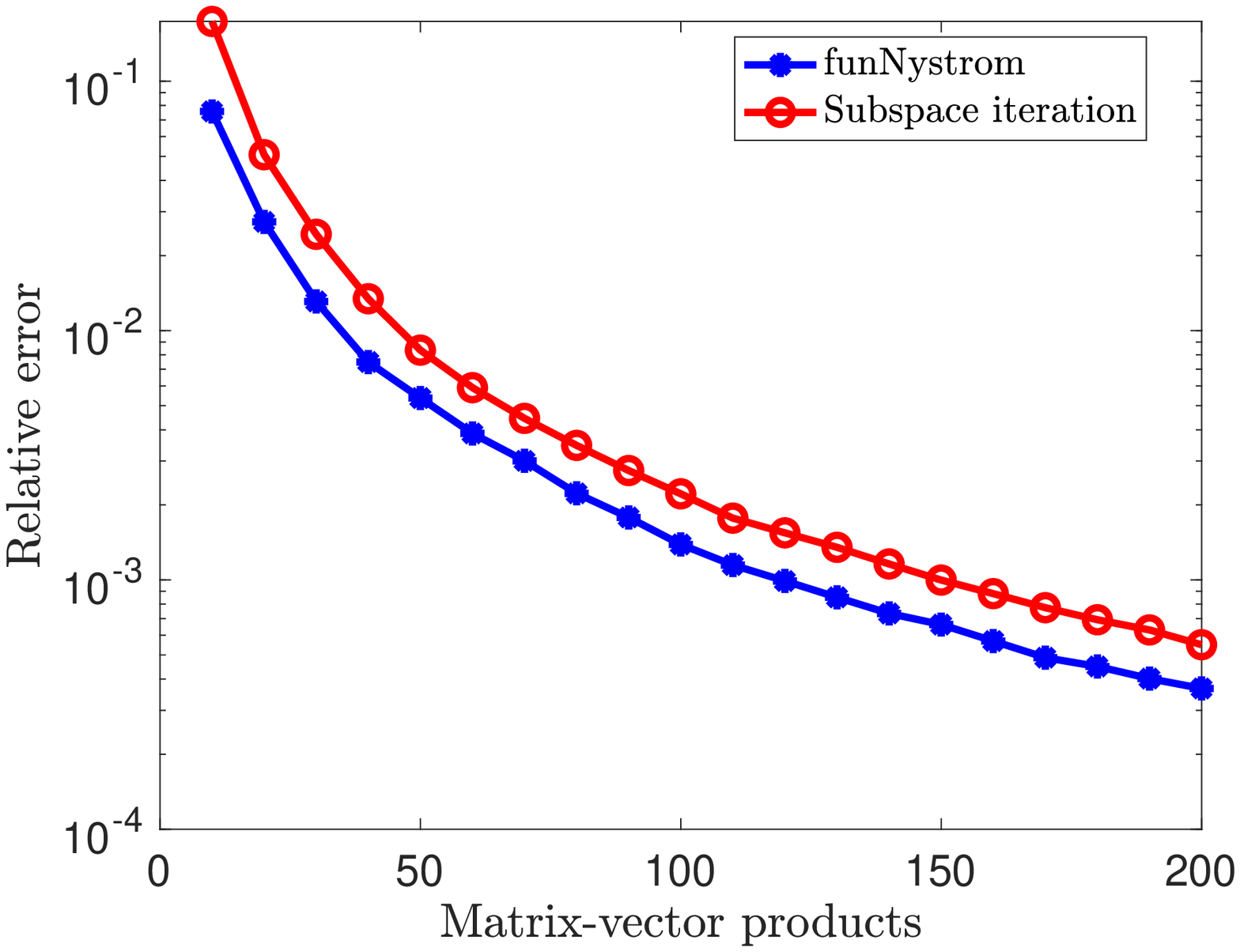}  
  \caption{$\bm{A}_{\text{alg}}$ defined in \eqref{eq:exponential_decay} with $s = 100$ and $c = 2$. }
\end{subfigure}
\begin{subfigure}{.5\textwidth}
  \centering
  \includegraphics[width=\linewidth]{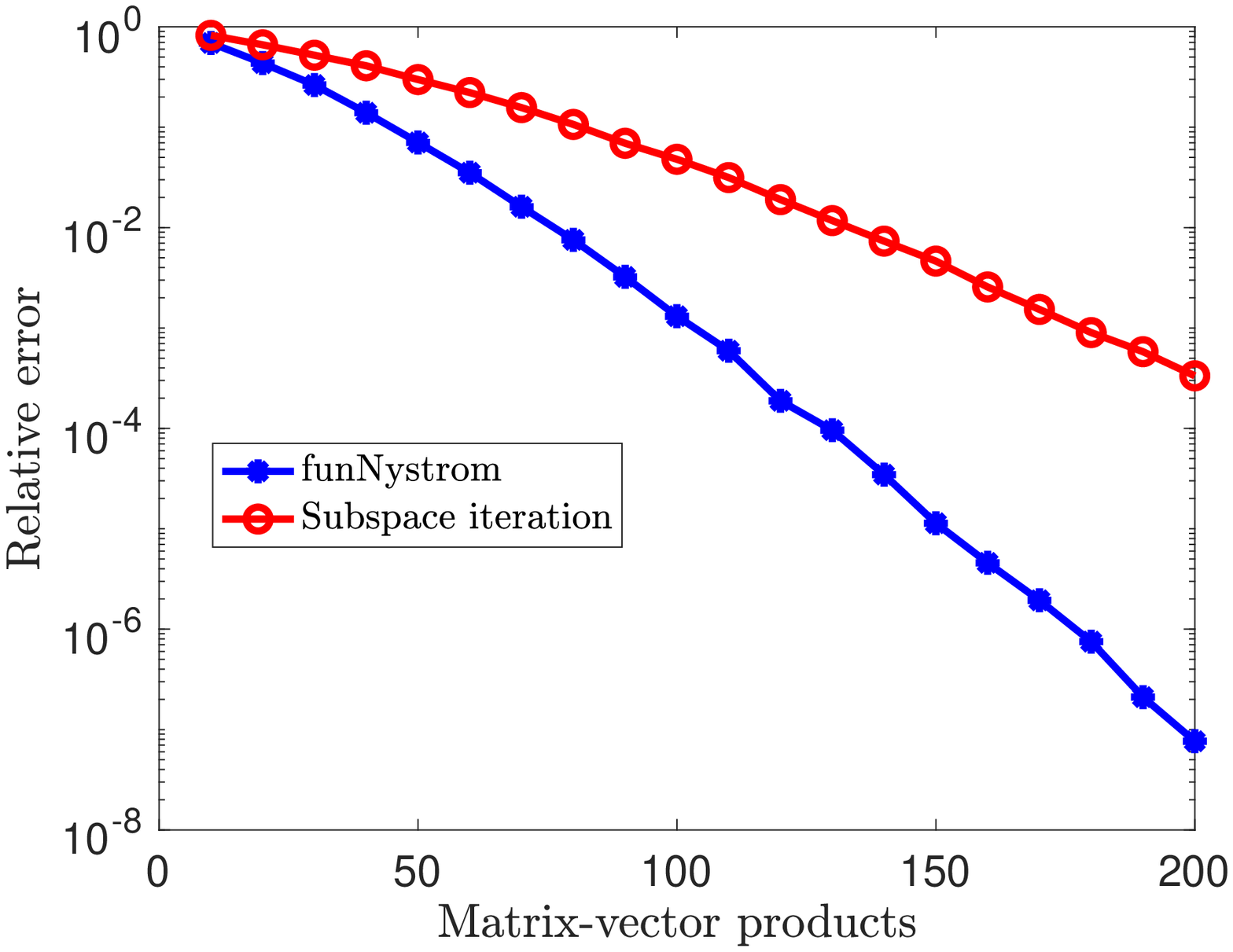}  
  \caption{$\bm{A}_{\text{exp}}$ defined in \eqref{eq:exponential_decay} with $s = 100$ and $\gamma = 0.9$.}
\end{subfigure}

\caption{Approximation of $\tr(\log(\bm{I}+\bm{A}))$ using Algorithm~\ref{alg:funnystrom} (blue) and the method presented in \cite{saibaba} (red). The x-axis represents the number of mvps performed with $\bm{A}$ to obtain the approximation, and the y-axis represents the relative error of the approximation.}
\label{fig:trace_est}
\end{figure}

\subsection{funHutch++}


A popular way to approximate the trace of an SPSD  matrix $\bm{B} \in \mathbb R^{n\times n}$ is to apply a stochastic estimator of the form
\begin{equation} \label{eq:hutchinson}
     \tr(\bm{B})\approx \tr_m(\bm{B}) := \frac{1}{m}\tr\left(\bm{\Phi}^T \bm{B} \bm{\Phi}\right)
\end{equation}
for an $n\times m$ random matrix $\bm{\Phi}$; in the following, we assume that $\bm{\Phi}$ is Gaussian random. To achieve a relative error of $\varepsilon$ (that is, $|\tr(\bm{B}) - \tr_m(\bm{B})| \le \varepsilon \tr(\bm{B})$ with high probability) 
one generally needs to choose $m \sim \varepsilon^{-2}$; see \cite{cortinoviskressner, dudley2020monte,improvedbounds, roosta, ubarusaad, ubaru2017applications}. Hutch++ \cite[Algorithm 1]{hutchpp} reduces this number to $m \sim \varepsilon^{-1}$ by combining~\eqref{eq:hutchinson} with (randomized) low-rank approximation. More specifically, given a budget of $m$ mvps with $\bm{B}$, Hutch++\footnote{Hutch++ as presented in \cite{hutchpp} uses Rademacher matrices instead of Gaussian random matrices. In order to remain consistent with the rest of this paper we choose  all random matrices to be Gaussians, which incurs no significant difference in the theory or numerical performance of Hutch++.} spends $2m/3$ mvps to construct a rank-$m/3$ approximation $\widehat{\bm{B}}$ based on the randomized SVD \cite{rsvd}. The remaining $m/3$ mvps are used to estimate the trace of the difference $\bm{B} - \widehat{\bm{B}}$ using~\eqref{eq:hutchinson}.

In~\cite[Section 3]{AHutchpp} a variant called Nyström++ is presented, which -- instead of the randomized SVD -- constructs a rank-$m/2$ Nyström approximation $\widehat{\bm{B}}$ using $m/2$ mvps. The remaining $m/2$ mvps are used to 
computed the correction $\tr_{m/2}(\bm{B}-\widehat{\bm{B}})$.
Nyström++ retains the property of Hutch++ that only $m \sim  \varepsilon^{-1}$ mvps are needed to obtain a relative error of $\varepsilon$, and it often provides better numerical performance. We refer to~\cite{chen,jiang2021optimal} for further variants of Hutch++.

When $\bm{B} = f(\bm{A})$ for SPSD $\bm{A}$ and a monotonically increasing function $f$ satisfying $f(0) = 0$, we can derive a cheaper version of Nyström++ by letting $\widehat{\bm{B}} = f(\widehat{\bm{A}}_{q,r})$ and thus considering the approximation
\begin{equation*}
    \tr(\bm{B}) \approx \tr(f(\widehat{\bm{A}}_{q,r})) +  \tr_{\ell}(f(\bm{A})-f(\widehat{\bm{A}}_{q,r}));
\end{equation*}
see also Algorithm~\ref{alg:funhpp}. The first part of this approximation bypasses the need for performing mvps with $f(\bm{A})$. \begin{algorithm}[t]
\caption{funNystr\"om++}
\label{alg:funhpp}
\textbf{input:} SPSD $\bm{A} \in \mathbb{R}^{n \times n}$. Rank $r$. Number of subspace iterations $q$. Number of samples in stochastic trace estimator $\ell$. Increasing function $f:[0,\infty) \mapsto [0,\infty)$ satisfying $f(0)=0$.\\
\textbf{output:} An approximation $\tr_{r,\ell}^{\mathsf{fn++}}(f(\bm{A})) \approx \tr(f(\bm{A}))$.
\begin{algorithmic}[1]
    \State Obtain a low rank approximation $f(\widehat{\bm{A}}_{q,r})$ using Algorithm~\ref{alg:funnystrom}. \label{line:lowrank}
    \State Compute $t_1 = \tr(f(\widehat{\bm{A}}_{q,r}))$.
    \State Compute $t_2 = \tr_{\ell}(f(\bm{A}) - f(\widehat{\bm{A}}_{q,r}))$.
    \State \textbf{return} $\tr_{r,\ell}^{\mathsf{fn++}}(f(\bm{A})) = t_1 + t_2$.
\end{algorithmic}
\end{algorithm}
The second part of the approximation still requires to perform or approximate such mvps via, e.g., the Lanczos method. 
Note that Algorithm~\ref{alg:funhpp} coincides with Nyström++ for $f(x) = x$. 
The following theorem extends the theoretical result on Hutch++ from~\cite[Theorem 1.1]{hutchpp} to Algorithm~\ref{alg:funhpp}.
\begin{theorem}\label{theorem:funnpp}
Let $f$ be operator monotone with $f(0) = 0$, $\bm{A}$ SPSD, and $\delta \in (0,1/2]$. Then one can choose $r = \ell \sim \sqrt{\log(\delta^{-1})}\varepsilon^{-1} + \log(\delta^{-1})$ such that the output of Algorithm~\ref{alg:funhpp} satisfies for any $q \geq 2$ the error bound
\begin{equation*}
    |\tr_{r,\ell}^{\mathsf{fn++}}(\bm{A})-\tr(f(\bm{A}))| \leq \varepsilon \tr(f(\bm{A}))
\end{equation*}
with probability at least $1-\delta$.
\end{theorem}
\begin{proof}
Let $r = 2k$ for $k > 4$. Applying Theorem~\ref{theorem:frobenius_opmon_probability}, with $p = k,t = e, u = \sqrt{2k}$ and using $\gamma \leq 1$, it follows that
\begin{equation*}
    \|f(\bm{A})-f(\widehat{\bm{A}}_{q,r})\|_F \leq 45\sqrt{\sum\limits_{i = k+1}^{n}f(\lambda_i)^2}
\end{equation*}
holds with probability at least $1-3e^{-k}$.
Thus, by \cite[Lemma 3.1]{hutchpp} we have with probability at least $1-3e^{-k}$ that \begin{equation*}
    \|f(\bm{A})-f(\widehat{\bm{A}}_{q,r})\|_F \leq \frac{45}{\sqrt{k}}\tr(f(\bm{A})).
\end{equation*}
Following the arguments in the proof of~\cite[Theorem 3.4]{AHutchpp}, there are constants $c,C$ such that for
$r = \ell = 2k > c \log(\delta^{-1})$, we have
\begin{equation*}
    |\tr_{r,\ell}^{\mathsf{fn++}}(f(\bm{A})) - \tr(f(\bm{A}))| \leq C k^{-1} \sqrt{\log(\delta^{-1})} \tr(f(\bm{A}))
\end{equation*}
 with probability at least $1-\delta$.
Setting $k \geq \varepsilon^{-1}C\sqrt{\log(\delta^{-1})}$ completes the proof. 
\end{proof}
We compare Algorithm~\ref{alg:funhpp} with Nyström++ from \cite{AHutchpp}. For Algorithm~\ref{alg:funhpp} we let $\ell = 6, 12, \ldots,60$, $r = 60, 120,\ldots, 600$, and for Nyström++ we set $m = 12,24,\ldots,120$. We approximate mvps with $f(\bm{A})$ using 10 iterations of the Lanczos method. Hence, for both the Nyström++ algorithm and Algorithm~\ref{alg:funhpp} we perform $120,240, \ldots, 1200$ mvps with $\bm{A}$. The obtained results are presented in Figure~\ref{fig:funhpp}. They indicate that Algorithm~\ref{alg:funhpp} with $q = 1$ and $q = 2$ can perform significantly better than Nyström++. It is interesting to note that our algorithms performs well for $q = 1$ even though this choice is not covered by the result of Theorem~\ref{theorem:funnpp}.

\begin{figure}
\begin{subfigure}{.5\textwidth}
  \centering
  \includegraphics[width=\linewidth]{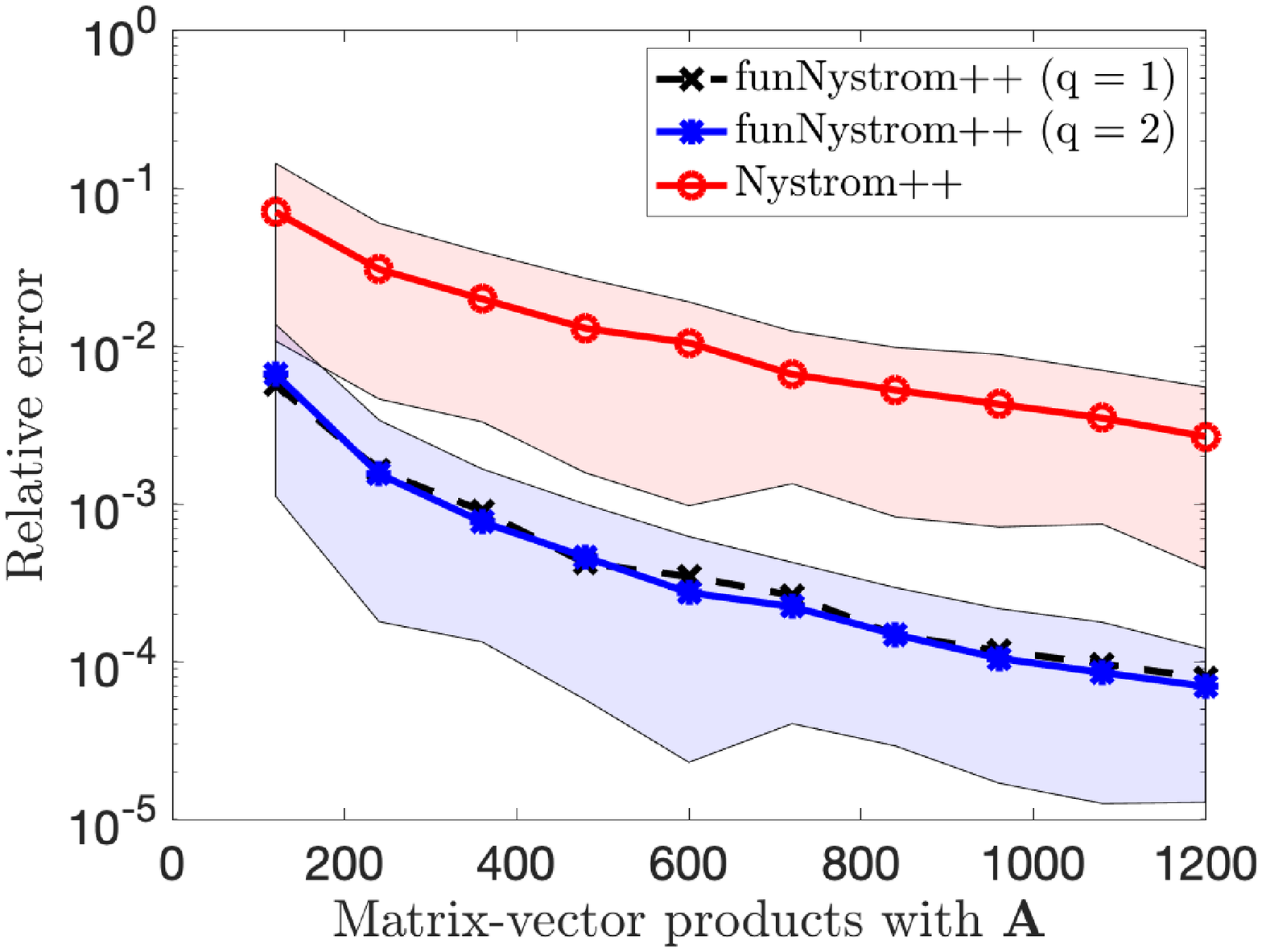}  
  \caption{$\bm{A}_{\text{alg}}$ defined in \eqref{eq:exponential_decay} with $s = 100$, $c=2$ and $f(x) = \log(1+x)$.}
\end{subfigure}
\begin{subfigure}{.5\textwidth}
  \centering
  \includegraphics[width=\linewidth]{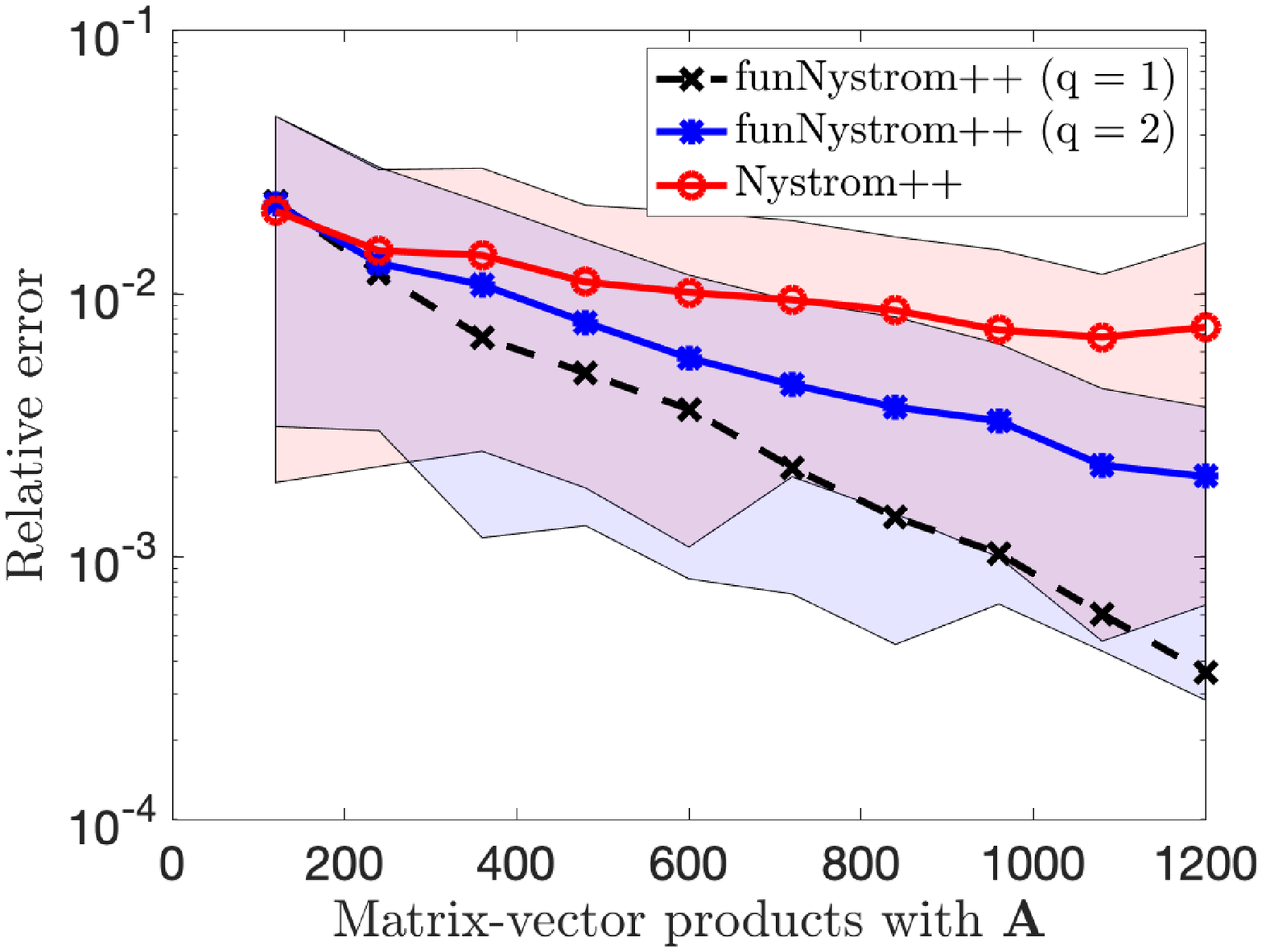}  
  \caption{$\bm{A}_{\text{exp}}$ defined in \eqref{eq:exponential_decay} with $s = 1, \gamma = e^{-\frac{1}{100}}$ and $f(x) = \frac{x}{x+0.1}$.}
\end{subfigure}

\caption{Approximation of $\tr(f(\bm{A}))$ using Algorithm~\ref{alg:funhpp} with $q = 1$ (black), Algorithm~\ref{alg:funhpp} with $q = 2$ (blue) and the Nyström++ algorithm presented in \cite{AHutchpp} (red). The x-axis represents the number of mvps performed with $\bm{A}$ to obtain the approximation, and the y-axis represents the relative error of the approximation. The shaded regions display the $5^{\text{th}}$ and $95^{\text{th}}$ percentiles.}
\label{fig:funhpp}
\end{figure}

\section{Conclusion}
funNystr\"om is a new and simple method for obtaining a low-rank approximation to 
$f(\bm{A})$ for an SPSD matrix $\bm{A}$. The experimental and theoretical evidence presented in this work seems to suggest that funNystr\"om is currently the method of choice for an operator monotone function $f:[0,\infty) \mapsto [0,\infty)$ with $f(0) = 0$. In contrast to standard randomized methods applied to $f(\bm{A})$, our method does not require exact or approximate matrix-vector products with $f(\bm{A})$. We have demonstrated that also other quantities associated with matrix functions, such as the trace, can be cheaply computed via funNystr\"om.

\begin{paragraph}{Acknowledgments} 
We thank the referees and Arvind 
Saibaba for helpful comments on this work.
\end{paragraph}

\bibliographystyle{siam}
\bibliography{bibliography}

\begin{appendices}

\section{Proof of Lemma~\ref{lemma:operator_norm_structural1}}\label{appendix:operator_norm}

To prove the structural bound of Lemma~\ref{lemma:operator_norm_structural1} for an arbitrary unitarily invariant norm $\|\cdot\|$, we make use of the following auxiliary result.

\begin{lemma}[{\cite[Theorem 2.1]{lee}}]\label{lemma:rotfeld}
Let $f : [0,\infty)\to [0,\infty)$ be concave. Then given a partitioned SPSD matrix $\begin{bmatrix} \bm{B} & \bm{X} \\ \bm{X}^T & \bm{C} \end{bmatrix}$ with square $\bm B$ and $\bm C$, one has
\begin{equation*}
    \left\|f\left(\begin{bmatrix} \bm{B} & \bm{X} \\ \bm{X}^T & \bm{C} \end{bmatrix}\right)\right\| \leq \|f(\bm{B})\| + \|f(\bm{C})\|.
\end{equation*}
\end{lemma}
\begin{proof}[Proof of Lemma~\ref{lemma:operator_norm_structural1}]
From~\eqref{eq:nystromproj} it follows that
\[
    \bm{A}-\hat{\bm{A}}_{q,k+p} =  \bm{A}^{1/2}(\bm{I}-\bm{P}_{\bm Y})\bm{A}^{1/2} = 
    \bm{U}\bm{\Lambda}^{1/2}(\bm{I} - \bm{P}_{\widetilde{\bm Y}}) \bm{\Lambda}^{1/2}\bm{U}^T,
\]
where we set $\bm Y = \bm{A}^{q-1/2}\bm{\Omega}$ and 
$\widetilde{\bm Y} = \bm{U}^T \bm Y$.
Combined with Lemma~\ref{lemma:loewner_trace}, this gives
\[
    \|f(\bm{A}) - f(\widehat{\bm{A}}_{q,k+p})\| \leq \|f(\bm{A}-\widehat{\bm{A}}_{q,k+p})\|
     = \|f(\bm{\Lambda}^{1/2}(\bm{I} - \bm{P}_{\widetilde{\bm Y}}) \bm{\Lambda}^{1/2})\|.
\]
As in the proof of Lemma~\ref{lemma:function_opmon_structural}, we set
$
    \bm{Z} = \widetilde{\bm Y} \bm{\Omega}_1^{\dagger} \bm{\Lambda}_1^{-(q-1/2)} = \begin{bmatrix} \bm{I} \\ \bm{F}\end{bmatrix}.
$
Using $\text{range}(\bm{Z}) \subseteq \text{range}(\widetilde{\bm Y})$, we obtain
$\bm{I}-\bm{P}_{\bm{Z}} \succeq \bm{I} - \bm{P}_{\widetilde{\bm Y}} \succeq \bm{0}$ and, in turn, $\bm{\Lambda}^{1/2}(\bm{I}-\bm{P}_{\bm{Z}})\bm{\Lambda}^{1/2} \succeq \bm{\Lambda}^{1/2}(\bm{I}- \bm{P}_{\widetilde{\bm Y}})\bm{\Lambda}^{1/2}\succeq \bm{0}$. Using 
Lemma~\ref{lemma:loewner_trace} (ii), this gives
\begin{equation} \label{eq:lalabound}
 \|f(\bm{\Lambda}^{1/2}(\bm{I}- \bm{P}_{\widetilde{\bm Y}})\bm{\Lambda}^{1/2})\|\leq \|f(\bm{\Lambda}^{1/2}(\bm{I}-\bm{P}_{\bm{Z}})\bm{\Lambda}^{1/2})\|.
\end{equation}
Exploiting
the $2\times 2$ block structure~\eqref{eq:projector_ineq} of the SPSD matrix $\bm{I}-\bm{P}_{\bm{Z}}$ and applying Lemma~\ref{lemma:rotfeld} yields
\begin{align*}
 &\|f(\bm{\Lambda}^{1/2}(\bm{I}-\bm{P}_{\bm{Z}}) \bm{\Lambda}^{1/2})\| \\
 \leq & \|f(\bm{\Lambda}_1^{1/2}(\bm{I}-(\bm{I}+\bm{F}^T \bm{F})^{-1})\bm{\Lambda}_1^{1/2})\|
    +\|f(\bm{\Lambda}_2^{1/2}(\bm{I}-\bm{F}(\bm{I}+\bm{F}^T \bm{F})^{-1}\bm{F}^T)\bm{\Lambda}_2^{1/2})\|.
\end{align*}
The proof is completed using the inequalities
\begin{align*}
 &\|f(\bm{\Lambda}_1^{1/2}(\bm{I}-(\bm{I}+\bm{F}^T \bm{F})^{-1})\bm{\Lambda}_1^{1/2})\| \le
 \|f(\bm{\Lambda}_1^{1/2}\bm{F}^T \bm{F} \bm{\Lambda}_1^{1/2})\| \\
 & \|f(\bm{\Lambda}_2^{1/2}(\bm{I}-\bm{F}(\bm{I}+\bm{F}^T \bm{F})^{-1}\bm{F}^T)\bm{\Lambda}_2^{1/2})\|  \le 
 \|f(\bm{\Lambda}_2)\|,
\end{align*}
which are derived from the inequalities in~\eqref{eq:projector_ineq}
with the same arguments used for~\eqref{eq:lalabound}.
\end{proof}

\section{Structural bounds for general Schatten norms} \label{app:schatten}

The Schatten-$s$ norm $\|\cdot\|_{(s)}$ for $1\le s \le \infty$ is a unitarily invariant norm defined as the $\ell^s$ norm of the vector of singular values of a matrix. It includes the Frobenius norm ($s = 2$), the nuclear norm ($s =1$), as well as the operator norm ($s=\infty$). Deriving a structural bound from the result of Lemma~\ref{lemma:operator_norm_structural1} for general $\|\cdot\|_{(s)}$ is not straightforward; the result for $s=2$ crucially depends on Lemma~\ref{lemma:probabilistic3}, for which we do not know an extension for general $s$. To circumvent this difficulty, the bound involves $f_+'(0)$, the right derivative of $f$ at $0$, which needs to be assumed finite. 
\begin{theorem}\label{theorem:schatten_norms}
Under Setting~\ref{setting}, assume that $\rank(\bm{\Omega}_1) = k$ and $f_+'(0) < \infty$. Then
\begin{align*}
    &\|f(\bm{A})-f(\widehat{\bm{A}}_{q,k+p})\|_{(s)} \leq \|f(\bm{\Lambda}_2)\|_{(s)} +\cdots  \\
    &  f'(0) \begin{cases}
        \|\bm{\Lambda}_2^{1/2} \bm{\Omega}_2 \bm{\Omega}_1^{\dagger}\|_{(2s)}^2 \quad &\text{ if } q = 1;\\
        \min\left\{\gamma^{2(q-1)} \|\bm{\Lambda}_2^{1/2} \bm{\Omega}_2 \bm{\Omega}_1^{\dagger}\|_{(2s)}^2 , \gamma^{q-3/2}\|\bm{\Lambda}_2 \bm{\Omega}_2 \bm{\Omega}_1^{\dagger}\|_{(s)}\right\} \quad &\text{ if } q \geq 2.
    \end{cases}
\end{align*}
\end{theorem}
\begin{proof}
Since $f$ is concave its maximal derivative is assumed at $0$. Together with $f(0) = 0$, this implies $\|f(\bm{B})\|_{(s)} \leq f'(0) \|\bm{B}\|_{(s)}$ for any square matrix $\bm B$. 
In particular,
\begin{align*}
    \|f(\bm{\Lambda}_1^{1/2}\bm{F}^T \bm{F} \bm{\Lambda}_1^{1/2})\|_{(s)} \leq & f'(0) \|\bm{\Lambda}_1^{1/2}\bm{F}^T \bm{F} \bm{\Lambda}_1^{1/2}\|_{(s)} \\
    =& f'(0)\|\bm{F} \bm{\Lambda}_1^{1/2}\|_{(2s)}^2 \leq f'(0)\gamma^{2(q-1)}\|\bm{\Lambda}_2^{1/2}\bm{\Omega}_2\bm{\Omega}_1^{\dagger}\|_{(2s)}^2.
\end{align*}
Applying Lemma~\ref{lemma:operator_norm_structural1} yields
\begin{equation} \label{eq:resultq1}
    \|f(\bm{A})-f(\widehat{\bm{A}}_{q,k+p})\|_{(s)} \leq \|f(\bm{\Lambda}_2)\|_{(s)} + f'(0)\gamma^{2(q-1)}\|\bm{\Lambda}_2^{1/2}\bm{\Omega}_2\bm{\Omega}_1^{\dagger}\|_{(2s)}^2,
\end{equation}
which already establishes the result for $q = 1$.

Now assume $q \geq 2$. Following the proof of Lemma~\ref{lemma:operator_norm_structural1} one shows
\begin{align*}
    \|f(\bm{A})-f(\widehat{\bm{A}}_{q,k+p})\|_{(s)} \leq &\|f(\bm{\Lambda}_2)\|_{(s)} \\
    + &\|f\left(\bm{\Lambda}_1^{1/2}(\bm{I}-(\bm{I}+\bm{F}^T \bm{F})^{-1})\bm{\Lambda}_1^{1/2}\right)\|_{(s)}.
\end{align*}
Note that the non-zero eigenvalues of $\bm{\Lambda}_1^{1/2}(\bm{I}-(\bm{I}+\bm{F}^T \bm{F})^{-1})\bm{\Lambda}_1^{1/2}$ and
\begin{equation*}
    (\bm{I}-\bm{P}_{\bm{Z}})\begin{bmatrix} \bm{\Lambda}_1 & \\ & \bm{0} \end{bmatrix} (\bm{I}-\bm{P}_{\bm{Z}})
\end{equation*}
are the same. Hence,
\begin{equation*}
    \|f\left(\bm{\Lambda}_1^{1/2}(\bm{I}-(\bm{I}+\bm{F}^T \bm{F})^{-1})\bm{\Lambda}_1^{1/2}\right)\|_{(s)} = \left\|f\left((\bm{I}-\bm{P}_{\bm{Z}})\begin{bmatrix} \bm{\Lambda}_1 & \\ & \bm{0} \end{bmatrix} (\bm{I}-\bm{P}_{\bm{Z}})\right)\right\|_{(s)}.
\end{equation*}
In turn,
\begin{align*}
    &\left\|f\left((\bm{I}-\bm{P}_{\bm{Z}})\begin{bmatrix} \bm{\Lambda}_1 & \\ & \bm{0} \end{bmatrix} (\bm{I}-\bm{P}_{\bm{Z}})\right)\right\|_{(s)} \leq f'(0) \left\|(\bm{I}-\bm{P}_{\bm{Z}})\begin{bmatrix} \bm{\Lambda}_1 & \\ & \bm{0} \end{bmatrix} (\bm{I}-\bm{P}_{\bm{Z}})\right\|_{(s)}\\
    \leq & f'(0)\left\|(\bm{I}-\bm{P}_{\bm{Z}})\begin{bmatrix} \bm{\Lambda}_1 & \\ & \bm{0} \end{bmatrix}\right\|_{(s)} \leq f'(0)\gamma^{q-3/2}\|\bm{\Lambda}_2\bm{\Omega}_2\bm{\Omega}_1^{\dagger}\|_{(s)},
\end{align*}
where the final inequality follows from \cite[Theorem 7]{saibaba_subspace}. Combined with the bound~\eqref{eq:resultq1}, this completes proof.
\end{proof}

Clearly, Theorem~\ref{theorem:schatten_norms} cannot be used when $f(x) = \sqrt{x}$. However, plugging $\|(\bm{\Lambda}_1^{1/2}\bm{F}^T\bm{F}\bm{\Lambda}_1^{1/2})^{1/2}\|_{(s)} = \|\bm{F}\bm{\Lambda}_1^{1/2}\|_{(s)}$ into the result of Lemma~\ref{lemma:operator_norm_structural1} for $f(x) = \sqrt{x}$ immediately gives
\begin{equation*}
    \|\bm{A}^{1/2} - \widehat{\bm{A}}_{q,k+p}^{1/2}\|_{(s)} \leq \|\bm{\Lambda}_2^{1/2}\|_{(s)} + \gamma^{q-1} \|\bm{\Lambda}_2^{1/2}\bm{\Omega}_2\bm{\Omega}_1^{\dagger}\|_{(s)}
\end{equation*}
as a structural bound.

\end{appendices}

\end{document}